\documentclass[11pt,reqno]{amsart}
\usepackage{fullpage}
\usepackage{times}
\usepackage{amsmath,amssymb,amsthm,url}
\usepackage[utf8]{inputenc}
\usepackage[english]{babel}
\usepackage{bbm}
\usepackage{enumerate}
\usepackage{enumitem}
\usepackage{bm}
\usepackage{graphicx}
\usepackage{mathrsfs}
\usepackage[mathscr]{euscript}
\usepackage[colorlinks=true, pdfstartview=FitH, linkcolor=blue, citecolor=blue, urlcolor=blue]{hyperref}
\usepackage{tikz-cd}
\usepackage{tabstackengine}
\stackMath
\usepackage{comment}
\usepackage[capitalise]{cleveref}
\usepackage{ushort}

\newtheorem{thm}{Theorem}[section]
\newtheorem{lem}[thm]{Lemma}
\newtheorem{prop}[thm]{Proposition}
\newtheorem{cor}[thm]{Corollary}
\newtheorem{ex}[thm]{Example}

\newtheorem{claim}[thm]{Claim}
\newtheorem{fact}[thm]{Fact}
\newtheorem{res}[thm]{Result}

\usepackage{mathtools}
\newenvironment{poc}{\begin{proof}[Proof of claim]}{\end{proof}}

\theoremstyle{definition}
\newtheorem{defn}[thm]{Definition}
\newtheorem{rem}[thm]{Remark}

\usepackage[capitalise]{cleveref}

\newcommand{\F}{\mathbb{F}}

\newcommand{\R}{\mathbb{R}} 
\newcommand{\N}{\mathbb{N}}
\newcommand{\Z}{\mathbb{Z}}
\newcommand{\Zp}{\Z_p}
\newcommand{\Zps}{\Z_{p^s}}

\newcommand{\cT}{\mathcal{T}}

\newcommand{\al}{\alpha}

\newcommand{\es}{\emptyset}
\newcommand{\GR}{\text{GR}}
\newcommand{\pr}{\prime}
\newcommand{\sm}{\setminus}

\newcommand{\lc}{\lceil}
\newcommand{\rc}{\rceil}

\newcommand{\Fpn}{\F_{p^n}}

\title{Generalized additive bases and difference bases for Cartesian product of finite abelian groups}
\author{Shuxing Li}
\address{Department of Mathematical Sciences \\ University of Delaware\\ Newark, DE 19716 \\ United States}
\email{shuxingl@udel.edu}
\author{Chi Hoi Yip}
\address{School of Mathematics\\ Georgia Institute of Technology\\ Atlanta, GA 30332\\ United States}
\email{cyip30@gatech.edu}
\begin{document}

\begin{abstract}
For a finite group $G$ and positive integer $g$, a $g$-additive basis is a subset of $G$ whose pairwise sums cover each element of $G$ at least $g$ times, with $g$-difference bases defined similarly using pairwise differences. While prior work focused on $1$-additive and $1$-difference bases, recent works of Kravitz and Schmutz--Tait explored $g$-additive and $g$-difference bases in finite abelian groups. This paper investigates such bases in $G^n$, the Cartesian product of a finite abelian group $G$. We construct $g$-additive and $g$-difference bases in $G^n$, which lead to asymptotically sharp upper bounds on the minimal sizes of such bases. Our proofs draw on ideas from additive combinatorics and combinatorial design theory.
\end{abstract}

\maketitle

\section{Introduction}
Given a subset $A$ of a group $G$, a natural question is how the multiset $\{\{ a_1+a_2 \mid a_1, a_2 \in A \}\}$ of pairwise sums and the multiset $\{\{ a_1-a_2 \mid a_1, a_2 \in A \}\}$ of pairwise differences compare with the ambient group $G$. This question connects to many areas in discrete mathematics, including Sidon sets in additive combinatorics  \cite{O04} as well as difference sets in combinatorial design theory \cite[Chapter VI]{BJL99v1}.    

We recall some standard notation in order to give precise definitions for Sidon sets and difference sets. Given a group $G$ and two subsets $A$, $B$ of $G$, we define the representation functions for each $x \in G$:
$$
r_{A+B}(x)=|\{(a,b) \in A \times B \mid a+b=x\}|, \quad r_{A-B}(x)=|\{(a,b) \in A \times B \mid a-b=x\}|.
$$
A subset $A$ of $G$ is a \emph{Sidon set} in $G$ if $r_{A+A}(x) \le 2$ for each $x \in G$. Equivalently, the sums generated from $A$ is maximally distinctive: if $a_1+a_2=a_3+a_4$ for $a_1, a_2, a_3, a_4 \in A$, then we must have $\{a_1,a_2\}=\{a_3,a_4\}$. A subset $B$ of $G$ is a \emph{difference set} in $G$ if $r_{B-B}(x)$ is a constant $g \ge 1$ for each nonidentity $x \in G$. Equivalently, every nonidentity element of $G$ can be represented as differences of elements from $B$ in exactly $g$ ways.

Let $g$ be a positive integer. A subset $A$ of $G$ is a \emph{$g$-additive basis for $G$} if $r_{A+A}(x)\geq g$ for all $x\in G$. Analogously, $B$ is a \emph{$g$-difference basis for $G$} if $r_{B-B}(x)\geq g$ for all $x\in G$. Observe that a $2$-additive basis is the counterpart of a Sidon set: every group element has at least two representations as a sum of elements of a $2$-additive basis, whereas for a Sidon set, each group element has at most two such representations. Moreover, a $g$-difference basis relaxes the notion of a difference set: every non-identity group element has at least, rather than exactly, $g$ representations as a difference of elements of a $g$-difference basis. Note that what we call a $g$-difference basis is termed a $g$-difference set in \cite{K21, ST25}. Here we use the term $g$-difference basis to distinguish it from the difference set defined above, which has been intensively studied in combinatorial design theory \cite[Chapter~VI]{BJL99v1}.

Let $\nu_g(G)$ denote the minimum size of a $g$-additive basis for $G$, and let $\eta_g(G)$ denote the minimum size of a $g$-difference basis for $G$. Historically, extensive research has been conducted on $1$-additive basis and $1$-difference basis across many groups. The problem of estimating $\nu_1(G)$ for cyclic groups $G$ was first proposed by Schur, and Rohrbach \cite{R37} initiated the study of $\nu_1(G)$ for arbitrary groups in 1937. We note that in the literature, $1$-additive bases are more commonly referred to as bases; see for example \cite{BH91, KL92, R37, S55}. In this paper, we adopt the term additive basis, following \cite{GS80}, to emphasize the distinction from difference bases. The study of $1$-difference basis can be dated back to 1949, initiated by R\'{e}dei and R\'{e}nyi in the setting of expressing a set of integers $\{1,2,\ldots,n\}$ as pairwise differences of elements from a set \cite{RR49}. Since then, there has been a series of research on $1$-difference basis, including \cite{BG19c,BG19,BG19b,FKL88}. Notably, we have the following result concerning $1$-additive bases and $1$-difference bases in finite groups, which are not necessarily abelian.

\begin{res}
\label{res-general}
Let $G$ be a finite group. Then the following holds.
\begin{itemize}
\item[(1)] {\rm (Kozma and Lev, \cite[Theorem 2]{KL92})} $\nu_1(G) \le \frac{4}{\sqrt{3}}\sqrt{|G|}$.
\item[(2)] {\rm (Finkelstein, Kleitman, and Leighton, \cite[Theorem B]{FKL88})} $\eta_1(G) \le 3\sqrt{|G|}$.
\end{itemize}
\end{res}

The above results apply to all finite groups, and their proof relies on the classification of finite simple groups. When confined to specific families of finite groups, there has been a series of intensive research concerning $1$-additive basis and $1$-difference basis, including \cite{BG19c,BG19,BG19b,BH91}. Specifically, when $G$ is a finite abelian group, various upper bounds of $\nu_1(G)$ and $\eta_1(G)$ have been derived, depending on structural information of $G$. For instance, Bertram and Herzog proved that if $|G|$ is a square or $G$ is cyclic, then $\nu_1(G) < 2\sqrt{|G|}$ \cite[Lemma 8 and Proposition 13]{BH91}. 

In this paper, we study $g$-additive basis and $g$-difference basis in finite abelian groups for general positive integers $g$. Employing a simple counting argument, we have the following trivial lower bounds of $\nu_g(G)$ and $\eta_g(G)$.

\begin{lem}
\label{lem-lb}
Let $g$ be a positive integer and $G$ be a finite abelian group. We have
$$
\nu_1(G) \ge \sqrt{2|G|+\frac{1}{4}}-\frac{1}{2}, \quad \text{ and } \quad \nu_g(G) \ge \sqrt{g|G|} \quad \text{   for   } g\geq 2;
$$
$$
\eta_g(G)\geq \sqrt{g(|G|-1)+\frac{1}{4}}+\frac{1}{2} \quad \text{   for   } g\geq 1.
$$
\end{lem}
Note that in \Cref{lem-lb}, the lower bounds of $\nu_1(G)$ and $\nu_2(G)$ almost coincide. That is because in an abelian group $G$, for any $a,b \in G$, $a+b$ and $b+a$ represent the same element in $G$. We note that lower bounds in \Cref{lem-lb} are well-known; see for example \cite{BG19, BH91, K21}. 

While the trivial lower bounds in \Cref{lem-lb} are not expected to be tight in general, it is natural to ask whether they become asymptotically tight in a suitable limiting regime. To this end, we consider the Cartesian product of a fixed finite abelian group and study the asymptotic behavior of $\nu_g(G^n)$ and $\eta_g(G^n)$ as $n\to\infty$, where $g$ is a fixed positive integer and $G$ is a fixed finite abelian group. Before stating our main results, we outline two motivations---drawn from the literature in additive combinatorics---for investigating $\nu_g(G^n)$ and $\eta_g(G^n)$ for a fixed $G$.

\subsection{Motivations from $g$-Sidon sets and $g$-difference basis}\label{sec:1.1}
Sidon sets are widely studied and there are hundreds of relevant papers; we refer to an annotated bibliography by O'Bryant \cite{O04}. For a positive integer $g$, a set $A$ within a group $G$ is a $g$-Sidon set if $r_{A+A}(x) \le g$ for each $x \in G$.  Let
$$
\beta_g(n)= \max\{ |A| \mid A \subset \{1,2,\ldots,n\}, \mbox{$A$ is a $g$-Sidon set in $\Z$} \}.
$$
A natural way to study $\beta_g(n)$ is through its asymptotic behavior. For $g\in\{2,3\}$, it is well known that $\lim_{n\to\infty} \beta_g(n)/\sqrt{n}=1$. However, for $g\ge 4$, the existence of the limit $\lim_{n\to\infty} \beta_g(n)/\sqrt{n}$ remains a challenging open problem \cite[Section~1.1]{CRV10}.

In 2010, Cilleruelo, Ruzsa, and Vinuesa \cite{CRV10} studied the quantity $\beta_g(n)/\sqrt{gn}$ and proved the following two double limits coincide:
\begin{equation}
\label{eqn-CRV}
\lim_{g\rightarrow \infty} \limsup_{n \rightarrow \infty} \frac{\beta_g(n)}{\sqrt{gn}} = \lim_{g\rightarrow \infty} \liminf_{n \rightarrow \infty} \frac{\beta_g(n)}{\sqrt{gn}}=\sigma,
\end{equation}
where $\sigma=\sup_{f\in \mathcal{E}} \int_{0}^{1} f(x) \,dx$ and 
\begin{equation}\label{eq:conv}
\mathcal{E}=\bigg\{f:[0,1] \to \R_{\geq 0}: \int_{\R} f(t)f(x-t) \, dt\leq 1 \text{ for all } x\in \R\bigg\}.
\end{equation} 
The key idea to the proof of \cref{eqn-CRV} is the construction of large $g$-Sidon sets in a cyclic group $\Z_q$. To this end, consider 
$\alpha_g=\limsup_{q\rightarrow \infty} \alpha_g(q)/\sqrt{q}$, where 
$$
\alpha_g(q)=\max\{ |A| \mid \mbox{$A$ is a $g$-Sidon set in $\Z_q$} \}
$$
for each positive integer $q$. Note that a trivial upper bound for $\alpha_g(q)$ is $\alpha_g(q) \le \sqrt{gq}$. They showed that $\alpha_g=\sqrt{g}+O(g^{3/10})$ \cite[Theorem 1.6]{CRV10}, which implies 
$$
\lim_{g \rightarrow \infty} \limsup_{q \rightarrow \infty} \alpha_g(q)/\sqrt{gq}=\lim_{g \rightarrow \infty} \alpha_g/\sqrt{g}=1.
$$ 

Inspired by the fascinating connection between the autoconvolution inequality in \cref{eq:conv} and $g$-Sidon sets, Barnard and Steinerberger~\cite{BS20} asked if a dual result holds where $g$-Sidon sets are replaced by $g$-difference bases, and \cref{eq:conv} is replaced by a corresponding one involving an analogous autocorrelation inequality. Kravitz \cite{K21} answered their question in the affirmative and proved several results that are analogous to those in Cilleruelo, Ruzsa, and Vinuesa \cite{CRV10}. More precisely, let $\eta_g(n)$ be the minimum size of a $g$-difference basis for $\{1,2,\ldots, n\}$; he proved the following.

\begin{res}[{Kravitz, \cite[Theorem 1.2]{K21}}]
\label{res:n}

\begin{equation}\label{eq:tau}
\lim_{g\rightarrow \infty} \limsup_{n \rightarrow \infty} \frac{\eta_g(n)}{\sqrt{gn}} = \lim_{g\rightarrow \infty} \liminf_{n \rightarrow \infty} \frac{\eta_g(n)}{\sqrt{gn}}=\tau,
\end{equation}
where $\tau=\inf_{f\in \mathcal{F}} \int_{\R} f(x)\,dx$ and 
\begin{equation}\label{eq:corr}
\mathcal{F}=\bigg\{f:\R \to \R_{\geq 0}: \int_{\R} f(t)f(x+t) \, dt\geq 1 \text{ for all } x\in [0,1]\bigg\}.
\end{equation} 
\end{res}
As a key ingredient to establish \cref{res:n}, he proved that for each positive integer $g$,
$$
\liminf_{n \to \infty} \frac{\eta_g(\Z_n)}{\sqrt{n}}=\sqrt{g}+O(g^{3/10}).
$$
In the same paper, he also proved the following result using the probabilistic method.

\begin{res} [{Kravitz, \cite[Section 6]{K21}}] \label{res:prob}
Let $g:\N \to \N$ satisfy $1\leq g(n)\leq n$ such that $n=o(\exp(O(g(n)))$ as $n\to \infty$. Then $\eta_{g(|G|)}(G)=(1+o(1))\sqrt{g(|G|) \cdot |G|}$ as $|G|\to \infty$. The same result also holds for $\nu_{g(|G|)}(G)$. 
\end{res}

Very recently, Schmutz and Tait \cite{ST25} showed that $\lim_{n \to \infty} \eta_{g}(n)/\sqrt{n}$ exists for each positive integer $g$, providing an affirmative answer to a question by Kravitz \cite[Question 7.4]{K21}.

\subsection{Motivations for studying Cartesian products of a finite abelian group}

In additive combinatorics, a growing body of work studies problems first formulated over the integers and then, in parallel, in the broader setting of finite abelian groups.  Compared with the interval of integers $\{1,2,\ldots,n\}$, finite abelian groups provide a more structured algebraic setting. Along these lines, a particularly useful framework is the finite field model \cite{Green05, Peluse24}, in which questions about $\{1,2,\ldots,n\}$ are reformulated in $\F_q^n$, the $n$-dimensional vector space over a fixed finite field $\F_q$. Going a step further, since requiring a vector-space structure is a strong assumption, it is natural to consider the analogous problems over $G^n$, where $G$ is a fixed finite abelian group.

For example, a celebrated theorem of Roth states that if a set $A \subseteq \{1,2,\ldots, n\}$ contains no 3-term arithmetic progression (3AP), then $|A|=o(n)$ as $n \to \infty$. By varying the ambient group, one can pose the same question over $G^n$, where $G$ is a fixed finite abelian group. In particular, the breakthroughs of Croot, Lev, and Pach \cite{CLP17} and Ellenberg and Gijswijt \cite{EG17} analyze the maximum size of 3AP-free sets in $\Z_4^n$ and $\F_q^n$, respectively.

We have discussed previous works related to $g$-difference bases and $g$-additive bases over the interval $\{1,2,\ldots, n\} \subset \Z$ and cyclic groups $\Z_n$. In \cite[Section 3]{ST25}, Schmutz and Tait studied $\eta_g(\Z_p^n)$, where $p$ is an odd prime. Following the same spirit as above, we study $\eta_g(G^n)$ and $\nu_g(G^n)$ for a fixed abelian group $G$. 

\subsection{Main results}

In this paper, we focus on the asymptotic behavior of $\nu_g(G^n)$ and $\eta_g(G^n)$, where $g$ is a fixed positive integer, $G$ is a fixed finite abelian group, and $n\to \infty$. We prove several results that are of similar flavor to those described in Section~\ref{sec:1.1}. Note that in Result~\ref{res:prob}, the number $g$ depends on $|G|$ and is required to be sufficiently large in relation to $|G|$, thus Result~\ref{res:prob} does not apply in our setting where $g$ is fixed.

By Lemma~\ref{lem-lb}, $\sqrt{g}|G|^{n/2}$ is a trivial lower bound for $\eta_g(G^n)$, as well as for $\nu_g(G^n)$ with $g \ge 2$. Our first result indicates that this bound behaves well in an appropriate limiting regime which can be viewed as an analogue of Result~\ref{res:n}.

\begin{thm}\label{thm-limitg}
 Let $g \ge 1$ and let $G$ be a finite abelian group with $|G|>1$. Then 
\begin{align*}
\limsup_{n \rightarrow \infty} \frac{\eta_{g}(G^{n})}{\sqrt{g}|G|^{n/2}}=1+O\bigg(\frac{1}{\sqrt{g}}\bigg).
\end{align*}
In particular, 
$$
\lim_{g \rightarrow \infty} \limsup_{n \rightarrow \infty} \frac{\eta_g(G^{n})}{\sqrt{g}|G|^{n/2}}=\lim_{g \rightarrow \infty} \liminf_{n \rightarrow \infty}  \frac{\eta_g(G^{n})}{\sqrt{g}|G|^{n/2}}=1.
$$
The same result also holds for $\nu_g(G^n)$.
\end{thm}

Our second main result concerns the existence of the limit $\lim_{n \rightarrow \infty} \eta_g(G^n)/(\sqrt{g}|G|^{n/2})$, where $g$ is a fixed positive integer. For a wide range of finite abelian groups $G$, we are able to show that the limit exists. In order to state the results properly, we need the following definition. 

\begin{defn}
(Weakly admissible $2$-groups and admissible $2$-groups) Let $G$ be a finite abelian $2$-group such that
\begin{equation}
\label{eqn-admissible}
G=\prod_{i=1}^t \Z_{2^{2s_i}}^{u_i} \times \prod_{j=1}^w \Z_{2^{2r_j+1}}^{v_j} \times \Z_2^{v},
\end{equation}
with $t, w, v \ge 0$, $u_i, v_j \ge 1$, $1 \le s_1 < s_2 < \cdots < s_t$, and $1 \le r_1 < r_2 < \cdots < r_w$. 
We say 
\begin{enumerate}
    \item $G$ is a \emph{weakly admissible $2$-group} if
$$\sum_{1\leq i \leq t, s_i \notin \{1,2\}}  u_i  +\sum_{j=1}^{w} v_j \geq v;$$
\item $G$ is an \emph{admissible $2$-group} if $$2 \bigg\lfloor \frac{\sum_{1\leq i \leq t, s_i \notin \{1,2\}}  u_i}{2}\bigg \rfloor +\sum_{j=1}^{w} v_j \geq v.$$
\end{enumerate}
We regard the trivial group as an admissible $2$-group. A finite abelian $2$-group that is not admissible is defined to be an \emph{inadmissible $2$-group}.
\end{defn} 

%\textcolor{red}{$\prod_{1\leq i \leq t, s_i \notin \{1,2\}} \Z_{2^{2s_i}}^{u_i}/\prod_{1\leq i \leq t, s_i \notin \{1,2\}} \Z_{2^{2(s_i-3)}}^{u_i}=\Z_{64}^{\sum_{i=1}^t u_i}$. $$\Z_{64}^{2\lfloor(\sum_{1\leq i \leq t, s_i \notin \{1,2\}} u_i)/2  \rfloor}/(\Z_2 \times Z_8)^{\lfloor(\sum_{1\leq i \leq t, s_i \notin \{1,2\}} u_i)/2  \rfloor}/\Z_4^{\lfloor(\sum_{1\leq i \leq t, s_i \notin \{1,2\}} u_i)/2  \rfloor}=\Z_8^{{2\lfloor(\sum_{1\leq i \leq t, s_i \notin \{1,2\}} u_i)/2  \rfloor}}$$}

By definition, an admissible $2$-group is necessarily weakly admissible. Indeed, in \cref{eqn-admissible}, if $\sum_{1\leq i \leq t, s_i \notin \{1,2\}}  u_i$ is even, then admissible and weakly admissible $2$-groups coincide. While the definition of weakly admissible and admissible $2$-groups seems to be fairly restrictive, with the help of the theory of partition \cite{EL41,F93}, we shall show in \cref{prop-admissible} that almost all finite abelian $2$-groups are admissible.

Utilizing the notion of weakly admissible $2$-groups, we have the following result.

\begin{thm}
\label{thm-weaklyadmissible}
Let $G=G_{2} \times H$, where $G_{2}$ is a weakly admissible $2$-group and $H$ has odd order.  Then $\eta_{1}(G^{2 n})=(1+o_n(1))|G|^{n}$
as $n\to \infty$. In particular, for each $g \ge 1$, both limits $\lim_{n \rightarrow \infty} \frac{\eta_{g}(G^{2 n})}{\sqrt{g}|G|^{n}}$ and $\lim _{n \rightarrow \infty} \frac{\eta_{g}(G^{2n+1})}{\sqrt{g}|G|^{n+\frac{1}{2}}}$ exist.
\end{thm}

This result generalizes \cite[Theorem 2]{ST25} by Schmutz and Tait, which concerns the special case $G=\Z_p$ with $p$ being an odd prime. Note that the two limits in \cref{thm-weaklyadmissible} may not coincide, therefore, the sequence $\big(\frac{\eta_{g}(G^{n})}{\sqrt{g}|G|^{n/2}} \big)_{n \ge 1}$ is not necessarily convergent. In view of this, \cref{thm-weaklyadmissible} can be strengthened using the notion of admissible $2$-groups. 

%In particular, if $G=G_2 \times G_2 \times H \times H$, where $G_{2}$ is a weakly admissible $2$-group and $H$ has odd order, we know that for all $g\geq 1$, the limit $\lim_{n \to \infty} \frac{\eta_g(G^n)}{|G|^{n/2}}$ exists. In the next result, we show that the $2$-group part $G_2 \times G_2$ can be replaced by any admissible $2$-group of square order, which strengthens \cref{thm-weaklyadmissible}.

\begin{thm}
\label{thm-admissible}
Let $G=G_{2} \times H \times H$, where $G_{2}$ is an admissible $2$-group of square order and $H$ has odd order. Then 
$
\eta_{1}(G^{n})=(1+o_n(1))|G|^{\frac{n}{2}}
$
as $n \to \infty$. In particular, for each $g \ge 1$, the limit $\lim_{n \to \infty} \frac{\eta_g(G^n)}{\sqrt{g}|G|^{n/2}}$ exists.
\end{thm}

\begin{rem}
\begin{itemize}
\item[(1)] In \cref{thm-weaklyadmissible} and \cref{thm-admissible}, if $H$ is the trivial group of order $1$, then \cref{thm-admissible} strengthens \cref{thm-weaklyadmissible} by imposing that $G_2$ is admissible with square order.
\item[(2)] The condition that $G_2$ is weakly admissible in \cref{thm-weaklyadmissible} and admissible in \cref{thm-admissible} cannot be dropped. Indeed, when $H$ is the trivial group of order $1$ and $G_2=\Z_2$, which is neither weakly admissible nor admissible, both \cref{thm-weaklyadmissible} and \cref{thm-admissible} fail: 
\begin{itemize}
\item[(2a)] For $G=G_2 \times H=\Z_2$, we have $\eta_1(G^{2n})=\eta_2(G^{2n}) \ge \sqrt{2}|G|^n$.
\item[(2b)] For $G=G_2 \times H \times H=\Z_2$, we have $\eta_1(G^{n})=\eta_2(G^{n}) \ge \sqrt{2}|G|^\frac{n}{2}$.
\end{itemize}
\end{itemize}
\end{rem}

Note that \cref{thm-weaklyadmissible} and \cref{thm-admissible} address only $g$-difference bases. Deriving similar results for $g$-additive bases appears more challenging. This aligns with the observation in \cite[Section~5.1]{ST25} that sums are harder to handle than differences. 
Indeed, the same approach in the proof of \cref{thm-weaklyadmissible} and \cref{thm-admissible} does not apply to $g$-additive bases since the lower bound in \cref{lem-lb} already indicates that $\nu_1(G^n)=(1+o_n(1))|G|^{\frac{n}{2}}$ does not hold any more.

\medskip

Below, we fix some notation to be used throughout and outline the organization of the paper.

\textbf{Notation.} For a positive integer $n$, let $\Z_n$ be the cyclic group of order $n$. For a positive integer $m$, we use $[m]$ to denote the set $\{1,2,\ldots,m\}$. For a subset $A$ of an additively written group $G$, we use $-A$ to denote the subset $\{-g \mid g \in A\}$. We write $p$ or $p_i$ for primes.

\textbf{Organization of the paper.}
In \cref{sec:prelim}, we prove some preliminary results. In \cref{sec:construction}, we discuss a few constructions of small $g$-difference bases and/or $g$-additive bases using Galois rings and relative difference sets. In \cref{sec:proofmain}, we provide the proof of our main results. In \cref{sec:conclusion}, we conclude by presenting three open questions.

\section{Preliminaries}\label{sec:prelim}

In this section, we describe some preliminary results. First, we mention the extremely helpful quotient and direct product constructions of additive bases and difference bases. Second, we show that almost all finite $2$-groups are admissible. 

\subsection{The quotient and direct product constructions}

\begin{lem}[Quotient construction]
\label{lem-quotient}
Let $g_1$ and $g_2$ be positive integers. Let $G$ be an abelian group and $H$ be a subgroup of $G$. 
\begin{itemize}
\item[(1)] Suppose $\{a_{1}, a_{2}, \ldots, a_{s}\}$ is a $g_1$-additive basis of $H$ and $\{b_{1}+H, b_{2}+H, \ldots, b_{t}+H\}$ is a $g_{2}$-additive basis of $G/H$. Then 
$\{a_{k}+b_{i} \mid 1 \leq k \leq s, 1 \leq i \leq t\}$
is a $g_{1} g_{2}$-additive basis of $G$. Consequently, $\nu_{g_1g_2}(G) \leq \nu_{g_{1}}(H) \nu_{g_{2}}(G / H)$.
\item[(2)] Suppose $\{a_{1}, a_{2}, \ldots, a_{s}\}$ is a $g_1$-difference basis of $H$ and $\{b_{1}+H, b_{2}+H, \ldots, b_{t}+H\}$ is a $g_{2}$-difference basis of $G/H$. Then 
$\{a_{k}+b_{i} \mid 1 \leq k \leq s, 1 \leq i \leq t\}$
is a $g_{1} g_{2}$-difference basis of $G$. Consequently, $\eta_{g_1g_2}(G) \leq \eta_{g_{1}}(H) \eta_{g_{2}}(G / H)$.
\end{itemize}
\end{lem}
\begin{proof}
We only prove part (2) as the proof of part (1) is analogous. For an arbitrary $x \in G$, regard $x+H$ as an element of $G/H$, then there exist $g_{2}$ pairs of $b_{i}+H$ and $b_{j}+H$, such that $x+H=(b_{i}+H)-(b_{j}+H)=(b_i-b_j)+H$.  
Given one such $b_{i}+H$ and $b_{j}+H$ pair, we know that $x=b_{i}-b_{j}+y$ for some $y \in H$. Therefore, there exist $g_{1}$ pairs of $a_{k}$ and $a_{\ell}$ in $H$, such that $y=a_{k}-a_{\ell}$. Consequently, $x=(a_{k}+b_{i})-(a_{\ell}+b_{j})$. We want to show that there exist $g_{1} g_{2}$ pairs of $(a_{k}+b_{i},a_{\ell}+b_{j})$, such that $x=(a_{k}+b_{i})-(a_{\ell}+b_{j})$. For this purpose, it suffices to show that $(a_{k}+b_{i},a_{\ell}+b_{j})=(a_{k^{\pr}}+b_{i^{\pr}},a_{\ell^{\pr}}+b_{j^{\pr}})$ implies 
$(i,j)=(i^{\pr},j^{\pr})$ and $(k,\ell)=(k^{\pr},\ell^{\pr})$. This is true as $a_k+b_i=a_{k^{\pr}}+b_{i^{\pr}}$ implies $(i,k)=(i^{\pr},k^{\pr})$ and $a_{\ell}+b_j=a_{{\ell}^{\pr}}+b_{j^{\pr}}$ implies $(j,\ell)=(j^{\pr},\ell^{\pr})$. Thus, $\{a_{k}+b_{j} \mid 1 \leq k \leq s, 1 \leq j \leq t\}$ forms a $g_{1} g_{2}$-difference basis of $G$.  
Consequently, $\eta_{g_{1} g_{2}}(G) \leq \eta_{g_{1}}(H) \eta_{g_{2}}(G / H)$.
\end{proof}

As a direct consequence of \cref{lem-quotient}, we have the following corollary.

\begin{cor}[Direct product construction]
\label{cor-directproduct}
Let $g_1$ and $g_2$ be positive integers. Let $G=G_{1} \times G_{2}$ be an abelian group. 
\begin{itemize}
\item[(1)] For $i \in \{1,2\}$, let $D_i$ be a $g_i$-additive basis of $G_i$. Then $D_1 \times D_2$ is a $g_1g_2$-additive basis $G_1 \times G_2$. Consequently, $\nu_{g_{1} g_{2}}(G_{1} \times G_{2}) \le \nu_{g_{1}}(G_{1}) \nu_{g_{2}}(G_{2})$.
\item[(2)] For $i \in \{1,2\}$, let $D_i$ be a $g_i$-difference basis of $G_i$. Then $D_1 \times D_2$ is a $g_1g_2$-difference basis $G_1 \times G_2$. Consequently, $\eta_{g_{1} g_{2}}(G_{1} \times G_{2}) \le \eta_{g_{1}}(G_{1}) \eta_{g_{2}}(G_{2})$.
\end{itemize}
\end{cor}

Next, we use \cref{cor-directproduct} to prove the following result, which plays a fundamental role in the study of the sequence $\big( \frac{\eta_{g}(G^{n})}{\sqrt{g}|G|^{n/2}} \big)_{n \ge 1}$. Indeed, the convergence of the sequence $\big(\frac{\eta_g(G^n)}{\sqrt{g}|G|^{n/2}}\big)_{n \ge 1}$ is intriguingly related to the asymptotic optimality of the sequence $\big(\frac{\eta_1(G^n)}{|G|^{n/2}}\big)_{n \ge 1}$.

\begin{prop}
\label{prop-limits}
Let $G$ be a finite abelian group. 
\begin{enumerate}
    \item Assume that $\eta_1(G^n)=(1+o_{n}(1))|G|^{n/2}$ as $n \to \infty$. Then for each $g\geq 1$, the limit $\lim_{n \to \infty} \frac{\eta_g(G^n)}{\sqrt{g}|G|^{n/2}}$ exists.
    \item Assume that $\eta_1(G^{2n})=(1+o_{n}(1))|G|^{n}$ as $n \to \infty$. Then for each $g\geq 1$, both limits $\lim_{n \to \infty} \frac{\eta_g(G^{2n})}{\sqrt{g}|G|^{n}}$ and $\lim_{n \to \infty} \frac{\eta_g(G^{2n+1})}{\sqrt{g}|G|^{n+1/2}}$ exist.
\end{enumerate}
\end{prop}
\begin{proof}
(1) Let $k$ be a fixed integer such that $|G|^k\geq g$. Then $G^k$ is a $g$-difference basis for the group $G^k$, so $\eta_g(G^k)\leq |G|^k$. For each positive integer $n\geq k$, by Corollary~\ref{cor-directproduct}(2), we have
$\eta_g(G^{n})\leq \eta_1(G^{n-k}) \eta_g(G^k).$ Since $\eta_1(G^n)=(1+o_{n}(1))|G|^{n/2}$ as $n \to \infty$,
\begin{equation}\label{eq-limsup}
\limsup_{n \rightarrow \infty} \frac{\eta_{g}(G^{n})}{\sqrt{g}|G|^{n/2}} \leq \frac{\eta_{g}(G^{k})}{\sqrt{g}|G|^{k/2}} \cdot \limsup_{n \rightarrow \infty} \frac{\eta_{1}(G^{n-k})}{|G|^{(n-k)/2}}=\frac{\eta_{g}(G^{k})}{\sqrt{g}|G|^{k/2}}.
\end{equation}
Note that \cref{eq-limsup} holds for each integer $k$ with $|G|^k\geq g$. Thus, we have
$$
\limsup_{n \rightarrow \infty} \frac{\eta_{g}(G^{n})}{\sqrt{g}|G|^{n/2}} \leq \liminf_{k \rightarrow \infty} \frac{\eta_{g}(G^{k})}{\sqrt{g}|G|^{k/2}}.
$$
This shows the limit $\lim_{n \to \infty} \frac{\eta_g(G^n)}{\sqrt{g}|G|^{n/2}}$ exists.

The proof of part (2) is very similar and hence omitted.
\end{proof}

\subsection{Almost all finite $2$-groups are admissible}

In the following, we proceed to show that almost all finite abelian 2-groups are admissible, and thus, weakly admissible. We first introduce some notations. Let $\mathcal{G}_n$ be the set of all nonisomorphic abelian groups of order $2^n$. Let
$$
\mathcal{A}_n=\{ G \in \mathcal{G}_n \mid \text{$G$ is admissible} \}
$$
be the set of all nonisomorphic admissible groups of order $2^n$. We shall use tools from the theory of partitions. A \emph{partition} of a positive integer $n$ is a way of writing $n$ as a sum of positive integers, and two sums that differ only in the order of their summands are considered as the same partition. Let $p(n)$ be the number of partitions of the integer $n$. 

As a preparation, we have the following observation.
\begin{lem}
\label{lem-pn}
There is a one-to-one correspondence between groups in $\mathcal{G}_n$ and partitions of $n$. Therefore, $|\mathcal{G}_n|=p(n)$. 
\end{lem}
\begin{proof}
Let $G$ be a finite abelian $2$-group with order $2^{n}$. By the fundamental theorem of finite abelian groups, there is a one-to-one correspondence between the invariant factors of $G$ and the partitions of $n$. Indeed, each partition $\sum_{i=1}^{t}ia_{i}$ of the integer $n$, where $i$ occurs $a_i \ge 0$ times, corresponds to a $2$-group $G=\prod_{i=1}^{t} \Z_{2^i}^{a_i}$. Therefore, $|\mathcal{G}_n|=p(n)$. 
\end{proof}

The following proposition justifies that almost all finite abelian 2-groups are admissible.

\begin{prop}
\label{prop-admissible}
Almost all finite abelian 2-groups are admissible. More precisely,
$$
\lim_{n \to \infty} \frac{\sum_{i=1}^n |\mathcal{A}_i|}{\sum_{i=1}^n |\mathcal{G}_i|}=1.
$$
\end{prop}
\begin{proof}
Each $G \in \mathcal{G}_n$ can be written as
$$
G=\prod_{i=1}^t \Z_{2^{2s_i}}^{u_i} \times \Z_{2^4}^{\ell_2} \times \Z_{2^2}^{\ell_1} \times \prod_{j=1}^w \Z_{2^{2r_j+1}}^{v_j} \times  \Z_2^{v},
$$
where $t, w, \ell_1, \ell_2, v \ge 0$, $u_i, v_j \ge 1$, $3 \le s_1 < s_2 < \cdots < s_t$, $1 \le r_1 < r_2 < \cdots < r_w$, and $n=\sum_{i=1}^t 2s_iu_i+4 \ell_2+2\ell_1+\sum_{j=1}^w(2r_j+1)v_j+v$. Moreover, $G \in \mathcal{A}_n$ if and only if $2M  +\sum_{j=1}^w v_j \ge v$, where $M=\lfloor \frac{\sum_{i=1}^t u_i}{2} \rfloor$. Equivalently, $G \in \mathcal{G}_n \sm \mathcal{A}_n$ if and only if $2M  +\sum_{j=1}^w v_j < v$. Therefore, $G \in \mathcal{G}_n \sm \mathcal{A}_n$ implies $k(G) \le 2M+1+\sum_{j=1}^w v_j +\ell_1+\ell_2+v \le 2v+\ell_1+\ell_2$, where $k(G)$ is the number of invariant factors of $G$, which is also the number of parts in the corresponding partition of $n$. By the correspondence in \Cref{lem-pn}, $|\mathcal{G}_n \sm \mathcal{A}_n|$ is at most the number of partitions $\lambda$ of $n$ such that $$k(\lambda)\leq 2X_{0}(\lambda)+X_{1}(\lambda)+X_2(\lambda),$$ where $k(\lambda)$ is the number of parts in $\lambda$, and for each $j\geq 0$, $X_j(\lambda)$ is the number of parts in $\lambda$ that equal to $j$. By Erd\"os and Lehner \cite[Theorem 1.1]{EL41}, $$k(\lambda)\geq \frac{\sqrt{n}\log n}{2\pi\sqrt{\frac{2}{3}}}$$
holds for $(1-o_n(1))p(n)$ partitions $\lambda$ of $n$, as $n \to \infty$. On the other hand, for each fixed $j\geq 0$, by Fristedt \cite[Theorem 2.1]{F93}, $X_j(\lambda)\leq \sqrt{n}\log \log n$ holds for $(1-o_n(1))p(n)$ partitions $\lambda$ of $n$, as $n \to \infty$. Thus, the number of partitions $\lambda$ of $n$ such that $k(\lambda)\leq 2X_{0}(\lambda)+X_{1}(\lambda)+X_2(\lambda)$ is $o(p(n))$, as $n\to \infty$. Equivalently, $|\mathcal{A}_n|=(1-o_n(1))|\mathcal{G}_n|$, as $n\to \infty$. Since $|\mathcal{G}_n|=p(n)\to \infty$ as $n\to \infty$, it follows that
\[
\lim_{n \to \infty} \frac{\sum_{i=1}^n |\mathcal{A}_i|}{\sum_{i=1}^n |\mathcal{G}_i|}=\lim_{n \to \infty} \frac{|\mathcal{A}_n|}{|\mathcal{G}_n|}=1.\qedhere
\]
\end{proof}

\section{Constructions based on Galois rings and relative difference sets}\label{sec:construction}

In this section, we present several constructions of $g$-additive bases and $g$-difference bases in finite abelian groups, which are based on Galois rings and relative difference sets. Let $A$ and $B$ be two subsets of group $G$. For simplicity, we say $A$ is a $g$-additive basis (resp. $g$-difference basis) of $B$ if $r_{A+A}(x) \ge g$ (resp. $r_{A-A}(x) \ge g$) for each $x \in B$. 

\subsection{Basics on Galois rings and relative difference sets}

We first recall some basics about Galois rings.

\begin{defn}[Galois ring]
Let $p$ be a prime and $s$, $n$ be positive integers. Let $f \in \Z_{p^s}[x]$ be a monic polynomial of degree $n$ such that the image of $f$ under the natural projection $\Z_{p^s}[x] \rightarrow \Zp[x]$ is irreducible over $\Zp$. Then the quotient ring $\Z_{p^s}[x]/(f(x))$ is called a \emph{Galois ring} of characteristic $p^s$ and rank $n$, denoted by $\GR(p^s,n)$.
\end{defn}

Galois rings are natural extensions of finite fields. Indeed, the Galois ring $\GR(p,n)$ is exactly the finite field $\Fpn$. The following fact follows from \cite[Theorem 14.8]{Wan}.

\begin{fact}
\label{fact-GR}
Let $p$ be a prime and $s$, $n$ be positive integers. Let $R$ be the Galois ring $\GR(p^s,n)$ deduced by degree $n$ monic polynomial $f(x) \in \Zps[x]$. 
For $0 \le i \le s$, let $p^iR$ be the principal ideal $(p^i)$ of $R$. The following holds true.
\begin{itemize}
\item[(1)] The additive group $(R,+)$ is isomorphic to $\Z_{p^s}^n$.
\item[(2)] There exists an element $\xi$ of $R$ with multiplicative order $p^n-1$, which is a root of $f(x)$. The set $\cT=\{0\} \cup \{ \xi^i \mid 0 \le i \le p^n-2\}$ is the \emph{Teichmuller system} of $R$.
\item[(3)] Each element of $R$ has a unique $p$-adic representation: 
$$
R=\{ a_0+a_1p+\cdots +a_{s-1}p^{s-1} \mid a_0, a_1,\ldots, a_{s-1} \in \cT \}.
$$  
Consequently, $\{a+pR \mid a \in \cT \}$ are distinct elements in the quotient ring $R/pR$ and the multiplicative group $U(R)$ of units of $R$ is:
$$
U(R)=\{ a_0+a_1p+\cdots +a_{s-1}p^{s-1} \mid a_0, a_1,\ldots, a_{s-1} \in \cT, a_0 \ne 0 \}=R \sm pR.
$$
In particular, $a-b \in U(R)$ for distinct $a,b \in \cT$.
\end{itemize}
\end{fact}

%Let $U(R)$ be , then $U(R)=R \sm pR$ and $|U(R)|=p^{sn}-p^{(s-1)n}$. Indeed, $U(R)=G_1 \times G_2$, where $G_1 \cong \Z_{p^n-1}$ and
%$$
%G_2 \cong \begin{cases}
%               \Z_{p^{s-1}}^n & \mbox{if $p$ odd, or $p=2$, $s \le 2$,} \\
%               \Z_2 \times \Z_{2^{s-2}} \times \Z_{2^{s-1}}^{n-2} & \mbox{if $p=2$, $s \ge 3$.}
%          \end{cases}
%$$
%For distinct elements $x,y \in G_1$, we have $x-y \in U(R)$. 

For a detailed treatment of Galois rings, we refer to \cite[Chapter XVI]{McD74} and \cite[Chapter 14]{Wan}. Apart from Galois rings, our construction of difference basis are also inspired by the configuration named \emph{relative difference sets} from design theory \cite[Chapter VI, Section 10]{BJL99v1}.   

\begin{defn}[Relative difference sets]
Let $G$ be a group of order $mn$ containing a subgroup $N$ of order $n$. A
$k$-subset $R$ of $G$ is called a $(m, n, k, \lambda)$ relative difference set with respect to the group $N$ if the multiset of nonzero differences $\{ g-g^{\pr} \mid g, g^{\pr} \in G, g \ne g^{\pr} \}$ consists of all elements in $G \sm N$ exactly $\lambda$ times and no element in $N$. The subgroup $N$ is called the forbidden subgroup.
\end{defn}

By definition, an $(m,n,k,\lambda)$ relative difference set $R$ in group $G$ relative to subgroup group $N$ forms a $\lambda$-difference basis of $G\sm N$. This observation indicates that relative difference sets serve as suitable initial structures in the construction of difference bases. Below, we list two classical families of relative difference sets that are useful for this purpose.

\begin{ex}
\label{ex-RDS}\
\begin{itemize}
\item[(1)] Let $p$ be an odd prime and $n \ge 1$. Then the subset $\{(x,x^2) \mid x \in \F_{p^n}\} \subseteq (\F_{p^n} \times \F_{p^n},+)$ is a $(p^n,p^n,p^n,1)$ relative difference set in $\Z_p^n \times \Z_p^n$ relative to $\{0\} \times \Z_p^n$.
\item[(2)] Let $R=\GR(4,n)$ and $\cT$ be the Teichmuller system of $R$. Then $\cT$ is a $(2^n,2^n,2^n,1)$ relative difference set in $\Z_4^n$ relative to $2\Z_4^n \cong \Z_2^n$. 
\end{itemize}
\end{ex}
For a comprehensive treatment on relative difference sets, we refer to \cite{P96,PSZ14}.

In \cref{subsec-podd}, we will use Galois rings to mimic the construction in \cref{ex-RDS}(1). On the other hand, the same construction does not apply for $p=2$. Indeed, the fact that the elementary abelian $2$-group $\Z_2^{2n}$ does not admit $(2^n,2^n,2^n,1)$ relative difference set complicates the construction of difference basis when $p=2$. Moreover, elementary abelian $2$-group truly stands out as $\eta_1(\Z_2^{2n})=\eta_2(\Z_2^{2n}) \ge 2^{n+\frac{1}{2}}$. We will address the construction for $p=2$ in \cref{subsec-peven}. In Section~\ref{sec:eta_g}, we utilize Galois rings to construct additive bases and difference bases with prescribed repetition number as sums and differences in a wide range of finite abelian groups. 

\subsection{Estimates on $\eta_1(G)$ with $G$ being $p$-groups for odd prime $p$}
\label{subsec-podd}

\begin{lem}
\label{lem-podd}
Let $p$ be an odd prime, $n \ge 1$, and $s \geq 1$. Then there exists a $1$-difference basis in $\Z_{p^s}^{2n}$ with size at most $p^{sn}+9p^{(s-\frac{1}{2})n}$. Consequently, for sufficiently large $n$, 
$
\eta_{1}(\Z_{p^s}^{2n})=(1+o_{n}(1)) p^{sn}.
$
\end{lem}
%{\color{red}better understand the construction below to bound $\nu_2(G)$. In the following construction, weaken $D_1$ so that the set $D_1 \times D_2$ also satisfies $D_1 \times D_2 \cap -(D_1 \times D_2)$ is empty} \textcolor{blue}{It seems the PCP construction gives the same result and also works when $p=2$. }

\begin{proof}
Our construction is inspired by the classical construction for Sidon sets due to Erd\"os and Tur\'an \cite[Section I]{ET41} as well as the \emph{generalized relative difference sets} \cite[Section 3.3]{LPS19}.

Let $R=\GR(p^{s}, n)$. Then $(R \times R,+)=\Z_{p^{s}}^{2n}$. Let $S=\{(x, x^{2}) \mid x \in R\}$. Then $S$ is a subset of $R \times R$ with $|S|=p^{sn}$. We claim that $S$ is a $1$-difference basis of $U(R) \times R$. For $(a,b) \in U(R) \times R$, the equation
$
(x_1,x_1^2)-(x_2,x_2^2)=(a,b)
$
has the unique solution 
$$
\begin{cases}
   x_1=\frac{a^{-1}b+a}{2}, \\
   x_2=\frac{a^{-1}b-a}{2}.
\end{cases}
$$

Note that $(R \times R) \sm (U(R) \times R)=pR \times R$. Recall that $(pR,+)=\Z_{p^{s-1}}^n$ and $(R,+)=\Z_{p^s}^n$. By \cref{res-general}(2), there exists a $1$-difference basis $D_1$ of $\Z_{p^{s-1}}^n$ with $|D_1| \le 3 p^{\frac{(s-1)n}{2}}$ and a $1$-difference basis $D_2$ of $\Z_{p^{s}}^n$ with $|D_2| \le 3 p^{\frac{sn}{2}}$. By \cref{cor-directproduct}(2), $D_1 \times D_2$ is a $1$-difference basis of $pR \times R$. 

Therefore, $S \cup (D_1 \times D_2)$ is a $1$-difference basis of $R \times R$ with 
$$
|S \cup (D_1 \times D_2)| \le |S|+|D_1||D_2| \le p^{sn}+(3p^{\frac{(s-1)n}{2}})(3p^{\frac{sn}{2}})=p^{sn}+9p^{(s-\frac{1}{2})n}.
$$
Therefore, for sufficiently large $n$, 
$\eta_{1}(\Z_{p^s}^{2n})=(1+o_{n}(1)) p^{sn}.$
\end{proof}

%\begin{rem}
%\textcolor{blue}{\cref{lem-podd} requires $p$ to be an odd prime. It cannot be directly generalized to the $p=2$ case, as the subset $S$ in the above proof is no longer a $1$-difference basis of $U(R) \times R$ when $p=2$. In the next section, we will resolve this issue by establishing an analogous result to \cref{lem-podd} with $p=2$.}    
%\end{rem}

Consequently, we can prove the following special case of \cref{thm-weaklyadmissible}.

\begin{cor}
\label{cor-oddgroup}
Let $G$ be a finite abelian group of odd order and $|G|>1$. For sufficiently large $n$,
$
\eta_1(G^{2n})=(1+o_n(1))|G|^n.
$
\end{cor}
\begin{proof}
By the fundamental theorem of finite abelian groups, $G$ can be written as
$G=\prod_{i=1}^t \Z_{p^{s_i}}^{u_i},$
where $t \ge 1$, $p_i^{s_i}$ are distinct odd prime powers, and $u_i \ge 1$ for $1 \le i \le t$. Thus, by \cref{cor-directproduct}(2) and \cref{lem-podd}, for sufficiently large $n$, we have
\[
\eta_1(G^{2n})\leq \prod_{i=1}^t \eta_1(\Z_{p^{s_i}}^{2nu_i})=(1+o_n(1))\prod_{i=1}^t p^{ns_iu_i}=(1+o_n(1))|G|^n.\qedhere
\]
\end{proof}

\subsection{Estimates on $\eta_1(G)$ with $G$ being $2$-groups}
\label{subsec-peven}

In this subsection, we derive analogues of \cref{cor-oddgroup} for finite abelian $2$-groups.

%As we shall see, the case of $2$-groups is much more nuanced than $p$-groups with $p$ being an odd prime ({\color{red} explain by analyzing $\Z_2^n$}). 

\begin{lem}
\label{lem-Z4n}
Let $n \ge 1$. Then $\eta_{1}(\Z_{4}^{n}) \leq 2^{n}+3 \cdot 2^{\frac{n}{2}}$. In particular, $\eta_{1}(\Z_{4}^{n})=2^{n}(1+o_n(1))$ as $n\to \infty$.
\end{lem}
\begin{proof}
Let $R=\GR(4,n)$. Then $(R,+)=\Z_4^n$. By \cref{ex-RDS}(2), the Teichmuller system $\cT$ is a $(2^{n}, 2^{n}, 2^{n}, 1)$ relative difference set in $\Z_{4}^{n}$ relative to $2\Z_{4}^{n} \cong \Z_{2}^{n}$.
%\cite[Theorem 1(b)]{BD97} 
Therefore, $\cT$ is a $1$-difference basis of $\Z_4^n \sm 2\Z_4^n$. Moreover, by \cref{res-general}(2), there exists a $1$-difference basis $S$ in $2\Z_{4}^{n} \cong \Z_2^n$ with $|S| \leq 3 \cdot 2^{\frac{n}{2}}$. Therefore, $\cT \cup S$ is a $1$-difference basis of $\Z_4^n$ with $|R \cup S| \leq 2^{n}+3 \cdot 2^{\frac{n}{2}}$. Consequently, $\eta_{1}(\Z_{4}^{n}) \leq 2^{n}+3 \cdot 2^{\frac{n}{2}}$. 
\end{proof}

Utilizing \cref{lem-Z4n}, we can construct $1$-difference bases in groups of the form $\Z_{2^{2s}}^n$.

\begin{lem}
\label{lem-Z22s}
Let $n \ge 4$ and $s \ge 1$. Then $\eta_1(\Z_{2^{2s}}^n)\leq 2^{sn}+(2^{s+1}-1)2^{(s-\frac{1}{2})n}$. In particular,
$\eta_{1}(\Z_{2^{2s}}^{n})=2^{sn}(1+o(1))$ as $n\to \infty$.
\end{lem}
\begin{proof}
We prove the lemma by induction on $s$. The base case $s=1$ follows from \cref{lem-Z4n}. Assume there exists a $1$-difference basis of size at most $2^{(s-1)n}+(2^s-1)2^{(s-\frac{3}{2})n}$ in $\Z_{2^{2(s-1)}}^n$. Note that $\Z_{2^{2s}}^n/\Z_{2^{2(s-1)}}^n \cong \Z_4^{n}$. Combining \cref{lem-quotient}(2) and \cref{lem-Z4n}, the induction hypothesis, we have 
\begin{align*}
\eta_1(\Z_{2^{2s}}^n) &\le \eta_1(\Z_{2^{2(s-1)}}^n) \eta_1(\Z_4^n) \le (2^{(s-1)n}+(2^s-1)2^{(s-\frac{3}{2})n})(2^{n}+3 \cdot 2^{\frac{n}{2}}) \\
&\le 2^{sn}+(2^s+2)2^{(s-\frac{1}{2})n}+3(2^s-1)2^{(s-1)n} 
%&\le 2^{sn}+(2^{s+1}-1)2^{(s-\frac{1}{2})n}+(3-2^s)2^{(s-\frac{1}{2})n}+3(2^s-1)2^{(s-1)n} \\
%&\le 2^{sn}+(2^{s+1}-1)2^{(s-\frac{1}{2})n}+(3(2^s-1)-(2^s-3)2^{\frac{n}{2}}) 2^{(s-1)n} \\
<2^{sn}+(2^{s+1}-1)2^{(s-\frac{1}{2})n},
\end{align*}
where the last inequality follows from $n \ge 4$. \qedhere
\end{proof}

In order to mimic the construction in \cref{lem-podd}, we consider a different operation defined on $R \times R$ with $R=\GR(4,n)$. For $n \ge 1$, let $R=\GR(4,n)$. Within $R \times R$, define an operation $*$ as 
$$
(x_1,y_1)*(x_2,y_2)=(x_1+x_2, y_1+y_2+x_1x_2),$$ where $x_1x_2$ is the usual multiplication in the Galois ring $R$ between $x_1$ and $x_2$. We note that the operation $*$ has been used to construct relative difference sets in semifields, see for example \cite[Section 4]{PSZ14}. Below, we show that $R \times R$ forms an abelian group under the operation $*$, denoted by $R * R=(R \times R, *)$.  

\begin{lem}
\label{lem-starop}
For $n \ge 1$ and $R=\GR(4,n)$, the group $R * R=(R \times R, *)$ is isomorphic to $\Z_{8}^{n} \times \Z_{2}^{n}$.
\end{lem}
\begin{proof}
It is routine to verify that $R*R$ forms an abelian group under the operation $*$ with identity $(0,0)$. For $i \geq 1$ and $(x, y) \in R * R,$ define
$$
(x, y)^{i}=\underbrace{(x, y) *(x, y) * \ldots *(x, y)}_{\text{$i$ terms}}.
$$
Note that for $(x, y) \in R * R$,
\begin{align*}
 (x, y)^{2}=(2 x, 2 y+x^{2}), \quad 
 (x, y)^{4}=(0,2 x^{2}), \quad
 (x, y)^{8}=(0,0).
\end{align*}
Thus, each element of $R*R$ has order divisible by $8$. Moreover, since $(1,0)$ has order $8$, then the exponent $\exp(R*R)=8$. In order to show that $R*R \cong \Z_8^n \times \Z_2^n$, it suffices to count the number of elements with order $1$, $2$, $4$, and $8$, respectively. 

\begin{itemize}
\item[(a)] The number of order $1$ elements is $1$, corresponding to the identity $(0,0)$.
\item[(b)] An element $(x,y)$ has order $2$ if and only if $(x,y) \ne (0,0)$ and $(x,y)^2=(2x,2y+x^2)=(0,0)$, which is equivalent to $(x,y) \in (2R \times 2R) \sm \{(0,0)\}$. Therefore, the number of order $2$ elements is $2^{2n}-1$.
\item[(c)] An element $(x,y)$ has order $4$ if and only if $(x,y) \notin 2R \times 2R$, and $(x,y)^4=(0,2x^2)=(0,0)$, which is equivalent to $(x,y) \in 2R \times (R \sm 2R)$. Therefore, the number of order $4$ elements is $2^n(2^{2n}-2^n)=2^{2n}(2^n-1)$.
\item[(d)] The number of order $8$ elements is $4^{2n}-1-(2^{2n}-1)-2^{2n}(2^n-1)=2^{4n}-2^{3n}$.
\end{itemize}
Combining the above counts and $\exp(R*R)=8$, we have $(R \times R, *) \cong \Z_{8}^{n} \times \Z_{2}^{n}$.
\end{proof}

%We remark that \cref{lem-starop} is a special case of a much more general result stated in \cite[Theorem 5.5]{BG19}. We include \cref{lem-starop} for the sake of completeness. 
Over $R*R$, we propose the following construction, which is analogous to that of \cref{lem-podd}.

\begin{lem}
\label{lem-Z8n2n}
Let $n \ge 1$. Then there exists a $1$-difference basis of $\Z_8^n \times \Z_2^n$ with size at most $4^n+3 \cdot 2^{\frac{3n}{2}}+9 \cdot 2^n$. For sufficiently large $n$, $\eta_{1}(\Z_{8}^{n} \times \Z_{2}^{n})=(1+o_{n}(1)) 4^{n}$.
\end{lem}
\begin{proof}
Let $R=\GR(4, n)$. Then by \cref{lem-starop}, $(R \times R, *) \cong \Z_{8}^{n} \times \Z_{2}^{n}$. For $(x,y) \in (R \times R, *)$, its inverse, denoted by $(x,y)^{-*}$, is equal to $(-x,x^2-y)$.

Let $S=\{(x, x^2) \mid x \in R\}$. Note that $(x, x^2)^{-*}=(-x, 0)$. We claim that $S$ is a $1$-difference basis of $(U(R) \times R, *)$ with $|S|=4^{n}$. Indeed, for $(a, b) \in U(R) \times R$, the equation
$$
(a, b)=(x_1,x_1^2)*(x_2,x_2^{2})^{-*}=(x_1, x_1^2)*(-x_2, 0)=(x_1-x_2, x_1^2-x_1x_2)
$$
has a unique solution
$$
\begin{cases}
x_1=a^{-1}b, \\
x_2=a^{-1}b-a.
\end{cases}
$$

Note that $2R \times R=(R \times R) \sm (U(R) \times R)$, and $(2R,+) \cong \Z_2^n$, $(R,+) =\Z_4^n$. We proceed to construct a $1$-difference basis of $(2R \times R, *)$, utilizing $1$-difference bases in $\Z_2^n$ and $\Z_4^n$. According to \cref{res-general}(2), there exists a $1$-difference basis $T$ in $(2R,+)$  with $|T| \leq 3 \cdot 2^{\frac{n}{2}}$. By \cref{lem-Z4n}, there exists a $1$-difference basis $W$ in $(R,+)$ with $|W| \leq 2^{n}+3 \cdot 2^{\frac{n}{2}}$. We claim that $T \times W$ forms a $1$-difference basis in $(2R \times R,*)$, with $|T \times W| \le 3 \cdot 2^{\frac{3n}{2}} + 9 \cdot 2^n$. For each $(a, b) \in 2 R \times R$, there exist $t_1, t_2 \in 2R$, such that $a=t_1-t_2$ and there exist $w_1, w_2 \in W$, such that $b=w_1-w_2$. Since $t_1, t_2 \in 2R$, we have $(t_2, w_2)^{-*}=(-t_{2},-w_2)$ and $t_1t_2=0$. It follows that
\begin{align*}
(t_1,w_1)*(t_2,w_2)^{-*}=&(t_1,w_1)*(-t_2,-w_2)=(t_1-t_2,w_1-w_2+t_1(-t_2)) \\
                        =&(t_1-t_2,w_1-w_2)=(a, b).
\end{align*}
Consequently, $S \cup (T \times W)$ forms a $1$-difference basis in $(R \times R, *)$ with $|S\cup(T \times W)| \le 4^{n}+3 \cdot 2^{\frac{3n}{2}} + 9 \cdot 2^n$. Therefore, for sufficiently large $n$, $\eta_{1}(\Z_{8}^{n} \times \Z_{2}^{n})=4^{n}(1+o_{n}(1))$.  
\end{proof}

Now we are ready to prove the following analogy to \cref{lem-podd}.

\begin{cor}
\label{cor-peven}
Let $n \ge 9$ and $s \ge 2$. Then there exists a $1$-difference basis in $\Z_{2^s}^{2n}$ with size at most $2^{sn}+(2^s-1)\cdot 2^{(s-\frac{1}{2})n}$. For sufficiently large $n$,
$
\eta_{1}(\Z_{2^s}^{2n})=2^{sn}(1+o_{n}(1)).
$
\end{cor}
\begin{proof}
Proof by induction on $s$. The base case $s=2$ follows from \cref{lem-Z4n}.  
For $s=3$, note that $\Z_{8}^{2n}/\Z_{4}^{n} \cong \Z_{8}^{n} \times \Z_{2}^{n}$. Then by \cref{lem-quotient}(2), \cref{lem-Z4n}, and \cref{lem-Z8n2n}, we have
\begin{align*}
\eta_{1}(\Z_{8}^{2n}) &\leq \eta_{1}(\Z_{8}^{n} \times \Z_{2}^{n}) \eta_{1}(\Z_{4}^{n})\leq (4^n+3 \cdot 2^{\frac{3n}{2}}+9 \cdot 2^n)(2^{n}+3 \cdot 2^{\frac{n}{2}}) \\
&=2^{3n}+6 \cdot 2^{\frac{5n}{2}}+18 \cdot 2^{2n}+27 \cdot 2^{\frac{3n}{2}} \le 2^{3n}+7\cdot 2^{\frac{5n}{2}},
\end{align*}
where the last inequality holds true if and only if $n \ge 9$. 

Assume for each $t$ with $3 \leq t<s$, we have $\eta_{1}(\Z_{2^t}^{2n}) \le 2^{tn}+(2^t-1)2^{(t-\frac{1}{2})n}$. Employing \cref{lem-quotient}(2) with $\Z_{2^s}^{2 n}/\Z_{4}^{2 n} \cong \Z_{2^{s-2}}^{2 n}$, we have
\begin{align*}
\eta_{1}(\Z_{2^s}^{2n}) &\leq \eta_{1}(\Z_{2^{s-2}}^{2n})\eta_{1}(\Z_{4}^{2n}) \leq (2^{(s-2)n}+(2^{s-2}-1)2^{(s-2-\frac{1}{2})n})(2^{2n}+3\cdot 2^\frac{3n}{2}) \\
&= 2^{sn}+(2^{s-2}+2)2^{(s-\frac{1}{2})n}+3(2^{s-2}-1)2^{(s-1)n} \le 2^{sn}+(2^s-1)2^{(s-\frac{1}{2})n}.
\end{align*}
Moreover, for sufficiently large $n$, $\eta_{1}(\Z_{2^s}^{2n})=2^{sn}(1+o_{n}(1))$.
\end{proof}

\begin{rem}
The condition $s \ge 2$ in \cref{cor-peven} is essential. For $s=1$, $\eta_1(\Z_2^{2n})=\eta_2(\Z_2^{2n}) \ge 2^{n+\frac{1}{2}}$. 
\end{rem}

Finally, we prove the following analogy of \cref{cor-oddgroup}. Note that a finite abelian $2$-group with all invariant factors divisible by $4$ is necessarily a weakly admissible $2$-group. The following result is a special case of \cref{thm-weaklyadmissible}. 

\begin{cor}
\label{cor-2group}
Let $G$ be a finite abelian $2$-group whose invariant factors are all divisible by $4$. Then for sufficiently large $n$, 
$
\eta_1(G^{2n})=(1+o_n(1))|G|^n.
$
\end{cor}
\begin{proof}
Since $G$ is a finite abelian $2$-group whose factors are all divisible by $4$, by the fundamental theorem of finite abelian groups, we can write $G=\prod_{i=1}^t \Z_{2^{s_i}}^{u_i},$
where $t \ge 1$, $2 \le s_1 < s_2 < \cdots <s_t$, and $u_i \ge 1$ for each $1 \le i \le t$. By \cref{cor-directproduct}(2) and \cref{cor-peven}, for sufficiently large $n$, 
\[
\eta_1(G^{2n}) \le \prod_{i=1}^t \eta_1(\Z_{2^{s_i}}^{2u_in})= (1+o_n(1))|G|^n.\qedhere
\]
\end{proof}

\subsection{Estimates on $\eta_g(G^{2n})$ and $\nu_g(G^{2n})$}\label{sec:eta_g}

In this subsection, we present explicit constructions of $g$-additive and $g$-difference bases in $G^{2n}$ for some special integers $g \ge 1$.

\begin{lem}\label{lem-PCPextension}
Let $G$ be a finite abelian group with $G=\prod_{i=1}^{m} \Z_{p_{i}^{s_i}}^{n_{i}}$, where $n_{i} \geqslant 1$ and $p_{i}^{s_{i}}$ being distinct prime powers. Suppose $m+2 \leq k\leq \min \{ p_{i}^{n_i}-1 \mid 1 \leq i \leq m\}$, then there exists a subset $U \subseteq G \times G$ with $|U|\leq k\sum_{I \subsetneq [m]} \prod_{i \in [m] \sm I} (p_i^{s_in_i}-1)$ such that $U$ is a $(k-m)(k-m-1)$-additive basis and a $(k-m)(k-m-1)$-difference basis of $G \times G$. 
\end{lem}
\begin{proof}
For each $1 \leq i \leq m$, let $R_{i}=\GR(p_{i}^{s_i}, n_{i})$ and $\cT_i$ be the Teichmuller system of $R_i$. Then $G=(\prod_{i=1}^{m} R_{i},+)$. For each element $(a,b) \in G \times G$, we can express it as $(a,b)=((a_{1}, a_{2}, \ldots, a_{m}),(b_{1}, b_{2}, \ldots, b_{m}))$, where $a_{i}, b_{i} \in R_{i}$ for $1\leq i \leq m$.

For each $1 \leq i \leq m$, since $m+2 \leq k \leq p_{i}^{n_i}-1$, by \cref{fact-GR}(3), there exists a subset $\{\alpha_{ij} \mid 1 \leq j \leq k\} \subseteq \cT_i \setminus \{0\} \subseteq U(R_i)$ such that $\alpha_{ij}-\alpha_{i\ell} \in U(R_{i})$ for all $1 \leq j, \ell \leq k$ with $j \neq \ell$. The following claim is the key to our construction.

\begin{claim}\label{claim:U_I}
For each proper subset $I$ of $[m]$, there exists $U_I \subseteq G \times G$ such that
\begin{enumerate}
    \item $U_I=-U_I$ and $|U_I| \le k\prod_{i\in [m]\setminus I} (p_{i}^{s_in_{i}}-1)$.
    \item For each $(a,b)=((a_{1}, a_{2}, \ldots, a_{m}),(b_{1}, b_{2}, \ldots, b_{m})) \in G \times G$ with $a_i,b_i \in R_i$ for $1\leq i \leq m$ and $a_i=b_i=0$ precisely when $i\in I$, we have
    $$
r_{U_I-U_I}((a,b)) \geq (k-m)(k-m-1).
$$
\end{enumerate}
\end{claim}
\begin{poc}
Let $|[m] \setminus I]=v$; then $v\geq 1$. Without loss of generality, by relabeling, we may assume that $[m] \setminus I=[v]$. For $1 \leq j \leq k$, define $S_j$ to be a subset of $G \times G$ such that
$$
S_{j}=\{((x_{1}, x_{2}, \ldots, x_{v}, 0, \ldots, 0),(\alpha_{1 j} x_{1}, \alpha_{2 j} x_{2}, \ldots, \alpha_{v j} x_{v},0, \ldots, 0)) \mid x_{i} \in R_{i}^{*}, 1 \leq i \leq v\}.
$$ 
Let $U_I=\bigcup_{j=1}^k S_j$; note that this is a disjoint union. Then 
$|U_I|=\sum_{j=1}^k|S_j| \le k\prod_{i\in [m]\setminus I} (p_{i}^{s_in_{i}}-1)$ and $U_I=-U_I$.

It remains to show that $U_I$ also satisfies (2). Let $(a,b)=((a_{1}, a_{2}, \ldots, a_{m}),(b_{1}, b_{2}, \ldots, b_{m})) \in G \times G$ with $a_i,b_i \in R_i$ for $1\leq i \leq m$ and $a_i=b_i=0$ precisely when $i>v$. Let $W:=W(a,b)=\{1\leq j \leq k \mid b_i = \alpha_{ij} a_i \text{ for some } 1\leq i \leq v\}$. Observe that $|W|\le v \le m$; indeed, for each $1\leq i \leq v$, these is at most one $j$ such that $b_i=\alpha_{ij} a_i$ if $a_i\neq 0$, and there is no $j$ with $b_i=\alpha_{ij} a_i$ if $a_i=0$ (since $b_i\neq 0$).

To show $r_{U_I-U_I}((a,b)) \geq (k-m)(k-m-1)$, it suffices to show that for each $j,\ell \in [k] \setminus W$ with $j\neq \ell$, we can find $u_j \in S_j$ and $u_{\ell} \in S_{\ell}$ such that $(a,b)=u_j-u_{\ell}$, equivalently, it suffices to show that for each $1\leq i \leq v$, there exist $x_i, y_i \in R_i^*$ such that 
\begin{equation}\label{eqn-SjSell}
a_i=x_i-y_i, \quad b_i=\alpha_{ij} x_i-\alpha_{\ell j} y_i    
\end{equation}
Indeed, since $\alpha_{ij}-\alpha_{i\ell}\in U(R_i)$, solving \cref{eqn-SjSell}, we obtain 
$$
x_{i}=\frac{b_i-\alpha_{i\ell}a_i}{\alpha_{ij}-\alpha_{i\ell}}, \quad 
y_{i}=\frac{b_i-\alpha_{ij} a_i}{\alpha_{ij}-\alpha_{i\ell}}
$$
with $x_i, y_i \in R_i^*$ since $j,\ell \notin W$, as required. 
\end{poc}    

For each proper subset $I$ of $[m]$, choose $U_I \subseteq G \times G$ as in Claim~\ref{claim:U_I}. Let $U=\bigcup_{I \subsetneq [m]} U_I$. Then $|U| \le k\sum_{I \subsetneq [m]} \prod_{i \in [m] \sm I} (p_i^{s_in_i}-1)$. For each $(a,b)=((a_{1}, a_{2}, \ldots, a_{m}),(b_{1}, b_{2}, \ldots, b_{m}))\in G \times G$ with $(a,b)\neq (\textbf{0},\textbf{0})$, we have
$$
r_{U-U}((a,b)) \geq r_{U_J-U_J}((a,b))\geq (k-m)(k-m-1),
$$
where $J=\{1\leq i \leq m \mid a_i=b_i=0\}$. On the other hand, $r_{U-U}((\textbf{0},\textbf{0}))=|U|\geq |U_{\emptyset}| = k \prod_{i \in [m] \sm I} (p_i^{s_in_i}-1) \geq k^2$. Thus, $U$ is a $(k-m)(k-m-1)$-difference basis of $G \times G$ with size at most $k\sum_{I \subsetneq [m]} \prod_{i \in [m] \sm I} p_i^{s_in_i}$, as required. Since $U_I=-U_I$ for each proper subset $I$ of $[m]$, it follows that $U=-U$. Thus, $U$ is a $(k-m)(k-m-1)$-additive basis of $G \times G$ as well.
\end{proof}

\begin{cor}
\label{cor-PCPextension}
Let $G$ be a finite abelian group with $G=\prod_{i=1}^{m} \Z_{p_{i}^{s_i}}^{n_{i}}$, where $n_{i} \geqslant 1$ and $p_{i}^{s_{i}}$ being distinct prime powers. Let $k\geq m+2$. Then as $n\to \infty$, 
\begin{align*}
\eta_{(k-m)(k-m-1)}(G^{2n}) \le k|G|^n(1+o_n(1)), \quad
\nu_{(k-m)(k-m-1)}(G^{2n}) \le k|G|^n(1+o_n(1)).
\end{align*}    
\end{cor}
\begin{proof}
We only prove the estimate for $\eta$; the proof for the estimate for $\nu$ is identical. Let $n$ be sufficiently large; then $m+2 \le k \leq n\min \{ p_{i}^{n_i}-1 \mid 1 \leq i \leq m\}$.
Applying \Cref{lem-PCPextension} to the group $G^{n}=\prod_{i=1}^m {p_i^{s_i}}^{n_in}$, there exists a $(k-m)(k-m-1)$-difference basis of $G^{2n}$ of size less than $k \sum_{I \subsetneq [m]} \prod_{i \in [m] \sm I} p_{i}^{s_in_in}$. Thus, as $n\to \infty$, 
\begin{align*}
\eta_{(k-m)(k-m-1)}(G^{2 n}) \leq &k \sum_{I \subsetneq [m]} \prod_{i \in [m] \sm I} p_{i}^{s_in_in}=k|G|^{n} \sum_{I \subsetneq [m]} \Big(\prod_{i \in I} \frac{1}{p_{i}^{s_in_{i}}} \Big)^{n}\\
=&k|G|^{n}\bigg(1+\sum_{\es \subsetneq I \subsetneq [m]} \Big(\prod_{i \in I} \frac{1}{p_{i}^{s_in_i}}\Big)^n\bigg)=k|G|^{n}(1+o_{n}(1)).\qedhere
\end{align*}
\end{proof}

The constructions in Section~\ref{sec:prelim} and this section can be refined and combined to give some improved and explicit upper bounds on $\nu_g(\Z_{p^s}^{2n})$ and $\eta_g(\Z_{p^s}^{2n})$. We refer to Appendix~\ref{appendix} for improved upper bounds and some specific examples.

\section{Proof of main results}\label{sec:proofmain}
In this section, we combine all the ingredients we have so far to prove our main results.

\subsection{Proof of \cref{thm-limitg}}
We only prove the limit of $\eta_g$; the proof for the limit of $\nu_g$ is identical. Below, we focus on $\limsup_{n \to \infty} \frac{\eta_g(G^n)}{\sqrt{g}|G|^{n/2}}$. By the fundamental theorem of finite abelian groups, each finite abelian group $G$ with $|G|>1$ can be written as $G=\prod_{i=1}^{m} \Z_{p_{i}^{s_i}}^{n_i}$, where $p_{i}^{s_i}$ are distinct prime powers and $n_i \ge 1$ for each $1 \leq i \leq m$. 

Given an integer $g \ge 1$. We first consider $\eta_g(G^{2n})$. Let $k$ be the smallest integer such that $(k-m)(k-m-1) \geq g$; then we have $k=\sqrt{g}+O(1)$. By \Cref{cor-PCPextension}, as $n \to \infty$,
\begin{equation}
\label{eqn-thm-limitg-4}
\eta_{g}(G^{2n}) \le \eta_{(k-m)(k-m-1)}(G^{2n}) \le k|G|^{n}(1+o_{n}(1))= (\sqrt{g}+O(1))(1+o_{n}(1))|G|^n.
\end{equation}

Next we consider $\eta_{g}(G^{2 n+1})$. By \cref{cor-directproduct}(2),
\begin{equation}
\label{eqn-thm-limitg-5}
\eta_{g}(G^{2n+1}) \leq \eta_{|G|}(G) \eta_{\lc \frac{g}{|G|} \rc}(G^{2n})=|G|\eta_{\lc \frac{g}{|G|} \rc}(G^{2n}).
\end{equation}
By \cref{eqn-thm-limitg-4}, for sufficiently large $n$,
\begin{equation}
\label{eqn-thm-limitg-6}
\eta_{\lc \frac{g}{|G|} \rc}(G^{2n}) \le \bigg(\sqrt{\frac{g}{|G|}}+O(1)\bigg)(1+o_n(1))|G|^n.
\end{equation}
Combining \cref{eqn-thm-limitg-5,eqn-thm-limitg-6}, for sufficiently large $n$,
$$
\eta_{g}(G^{2n+1}) \le \bigg(\sqrt{\frac{g}{|G|}}+O(1)\bigg)(1+o_n(1))|G|^{n+1}=(\sqrt{g}+O(1))(1+o_n(1))|G|^{n+1/2}.
$$
Therefore, we have shown
\begin{equation}
\label{eqn-thm-limitg-7}
\limsup_{n \to \infty} \frac{\eta_g(G^{n})}{\sqrt{g}|G|^{n/2}} \le 1+O\bigg(\frac{1}{\sqrt{g}}\bigg).
\end{equation}
Consequently, combining \cref{lem-lb} and \cref{eqn-thm-limitg-7}, we have
$$
\lim_{g \rightarrow \infty} \liminf_{n \rightarrow \infty} \frac{\eta_g(G^{n})}{\sqrt{g}|G|^{n/2}}=\lim_{g \rightarrow \infty} \limsup_{n \rightarrow \infty}  \frac{\eta_g(G^{n})}{\sqrt{g}|G|^{n/2}}=1.
$$

\subsection{Proof of \cref{thm-weaklyadmissible}}
Since $G_2$ is a weakly admissible $2$-group, by definition, we can write
$$
G_2=\Z_{2^2}^{v_2} \times \Z_{2^4}^{v_4} \times \Z_{2^5}^{v_5} \times \prod_{i=1}^{t} \Z_{2^{s_{i}}}^{u_i} \times \Z_{2^3}^{v_3} \times \Z_{2}^{v_1}, 
$$
where $v_1, v_2, v_3, v_4, v_5 \ge 0$, $t, u_i \ge 1$, $6 \leq s_{1} < s_{2} < \cdots < s_{t}$, and $v_3+v_5+M \geq v_1$ with $M=\sum_{i=1}^t u_i$.

Note that $$G_2^{2n}/\bigg(\Z_{2^2}^{2v_2n} \times \Z_{2^4}^{2v_4n} \times \Z_{2^2}^{2v_5n} \times \prod_{i=1}^{t} \Z_{2^{s_{i}-3}}^{2u_in}\bigg) \cong \Z_{2^3}^{2(v_3+v_5+M)n} \times \Z_2^{2v_1n}. $$ It follows from \cref{lem-quotient}(2) that
\begin{equation}
\label{eqn-weaklyadmissible-1}
\eta_1(G_2^{2n})\leq \eta_1\bigg(\Z_{2^2}^{2v_2n} \times \Z_{2^4}^{2v_4n} \times \Z_{2^2}^{2v_5n} \times \prod_{i=1}^{t} \Z_{2^{s_{i}-3}}^{2u_in}\bigg)\eta_1\bigg(\Z_{2^3}^{2(v_3+v_5+M)n} \times \Z_2^{2v_1n}\bigg).
\end{equation}
Note that $s_i-3 \ge 3$ for each $1 \le i \le t$. Combining \cref{cor-directproduct}(2) and \cref{cor-peven}, for sufficiently large $n$,
\begin{equation}
\label{eqn-weaklyadmissible-2}
\eta_1(\Z_{2^2}^{2v_2n} \times \Z_{2^4}^{2v_4n} \times \Z_{2^2}^{2v_5n} \times \prod_{i=1}^{t} \Z_{2^{s_{i}-3}}^{2u_in})=2^{2v_2n+4v_4n+2v_5n+\sum_{i=1}^t(s_i-3)u_in}(1+o_n(1)).
\end{equation}

Since $v_3+v_5+\sum_{i=1}^tu_i \ge v_1$, we have $$(\Z_{2^3}^{2(v_3+v_5+M)n} \times \Z_2^{2v_1n})/(\Z_{2^3}^{2v_1n} \times \Z_2^{2v_1n}) \cong \Z_{2^3}^{2(v_3+v_5+\sum_{i=1}^tu_i-v_1)n}.$$ It follows from \cref{lem-quotient}(2) that
\begin{equation}
\label{eqn-weaklyadmissible-3}
\eta_1(\Z_{2^3}^{2(v_3+v_5+M)n} \times \Z_2^{2v_1n})\leq \eta_1(\Z_{2^3}^{2v_1n} \times \Z_2^{2v_1n})\eta_1(\Z_{2^3}^{2(v_3+v_5+\sum_{i=1}^tu_i-v_1)n}).
\end{equation}
According to \cref{lem-Z8n2n,cor-peven}, for sufficiently large $n$,
\begin{align}
\label{eqn-weaklyadmissible-4}
\begin{split}
\eta_1(\Z_{2^3}^{2v_1n} \times \Z_2^{2v_1n})&=2^{4v_1n}(1+o_n(1)), \\
\eta_1(\Z_{2^3}^{2(v_3+v_5+\sum_{i=1}^tu_i-v_1)n})&=2^{3(v_3+v_5+\sum_{i=1}^tu_i-v_1)n}(1+o_n(1)).
\end{split}
\end{align}
Combining \cref{eqn-weaklyadmissible-1,eqn-weaklyadmissible-2,eqn-weaklyadmissible-3,eqn-weaklyadmissible-4}, for sufficiently large $n$, $\eta_{1}(G_2^{2n})=(1+o_{n}(1))|G_2|^{n}.$

Since $G=G_2 \times H$, where $H$ has odd order, by \cref{lem-podd} and \cref{cor-directproduct}(2), we have
\[
\eta_{1}(G^{2 n}) \le \eta_{1}(G_{2}^{2 n}) \eta_{1}(H^{2 n}) \leq(1+o_{n}(1))|G_{2}|^{n}(1+o_{n}(1))|H|^{n}=(1+o_{n}(1))|G|^{n}.
\]
Since $\eta_1(G^{2n}) \ge |G|^n$, then $\eta_{1}(G^{2 n})=(1+o_{n}(1))|G|^{n}$ as required. Finally, the existence of the limits $\lim_{n \rightarrow \infty} \frac{\eta_{g}(G^{2 n})}{\sqrt{g}|G|^{n}}$ and $\lim _{n \rightarrow \infty} \frac{\eta_{g}(G^{2n+1})}{\sqrt{g}|G|^{n+\frac{1}{2}}}$ then follows from \cref{prop-limits}.

\subsection{Proof of \cref{thm-admissible}}
Since $G_2$ is an admissible $2$-group, by definition, we can write
$$
G_2=\prod_{i=1}^t \Z_{2^{2s_i}}^{u_i} \times \Z_{2^4}^{\ell_2} \times \Z_{2^2}^{\ell_1} \times \prod_{j=1}^w \Z_{2^{2r_j+1}}^{v_j} \times \Z_{2^3}^{h_2} \times\Z_2^{h_1},
$$
where $t, w, \ell_1, \ell_2, h_1, h_2 \ge 0$, $u_i, v_j \ge 1$, $3 \le s_1 < s_2 < \cdots < s_t$, $2 \le r_1 < r_2 < \cdots < r_w$, and $2 \lfloor \frac{\sum_{i=1}^t u_i}{2} \rfloor +\sum_{j=1}^w v_j +h_2 \ge h_1$. Let $M=\lfloor \frac{\sum_{i=1}^t u_i}{2} \rfloor$.

Define 
$$
x=\sum_{i=1}^t u_i-2 M=\begin{cases}
    0 & \mbox{if $\sum_{i=1}^t u_i$ even,} \\
    1 & \mbox{if $\sum_{i=1}^t u_i$ odd.}
\end{cases}
$$
Note that $$G_2^{n}/(\Z_{2^{2(s_1-3)}}^{(u_1-x)n} \times \prod_{i=2}^t \Z_{2^{2(s_i-3)}}^{u_in} \times \prod_{j=1}^w \Z_{2^{2(r_j-1)}}^{v_jn}) \cong \Z_{2^{2s_1}}^{xn} \times \Z_{2^6}^{2 M n} \times \Z_{2^4}^{\ell_2n} \times \Z_{2^2}^{\ell_1n} \times \Z_{2^{3}}^{(\sum_{j=1}^w v_j+h_2)n} \times\Z_2^{h_1n}.$$ By \cref{lem-quotient}(2),
\begin{align}
\begin{split}
\label{eqn-admissible-1}
\eta_1(G_2^{n}) \le &\eta_1(\Z_{2^{2(s_1-3)}}^{(u_1-x)n} \times \prod_{i=2}^t \Z_{2^{2(s_i-3)}}^{u_in} \times \prod_{j=1}^w \Z_{2^{2(r_j-1)}}^{v_jn}) \\
& \cdot \eta_1(\Z_{2^{2s_1}}^{xn} \times \Z_{2^6}^{2 M n} \times \Z_{2^4}^{\ell_2n} \times \Z_{2^2}^{\ell_1n} \times \Z_{2^{3}}^{(\sum_{j=1}^w v_j+h_2)n} \times\Z_2^{h_1n}).
\end{split}
\end{align}
Note that $s_i-3 \ge 0$ for each $1 \le i \le t$ and $r_j-1 \ge 1$ for each $1 \le j \le w$. Combining \cref{cor-directproduct}(2) and \cref{lem-Z22s}, for sufficiently large $n$,
\begin{equation}
\label{eqn-admissible-2}
\eta_1(\Z_{2^{2(s_1-3)}}^{(u_1-x)n} \times \prod_{i=2}^t \Z_{2^{2(s_i-3)}}^{u_in} \times \prod_{j=1}^w \Z_{2^{2(r_j-1)}}^{v_jn})=2^{((s_1-3)(u_1-x)+\sum_{i=2}^t (s_i-3)u_i+\sum_{j=1}^w (r_j-1) v_j)n}(1+o_n(1)).
\end{equation}
Note that 
\begin{align*}
&(\Z_{2^{2s_1}}^{xn} \times \Z_{2^6}^{2 M n} \times \Z_{2^4}^{\ell_2n} \times \Z_{2^2}^{\ell_1n} \times \Z_{2^{3}}^{(\sum_{j=1}^w v_j+h_2)n} \times\Z_2^{h_1n}) / (\Z_{2^{2s_1}}^{xn} \times \Z_{2^3}^{2 M n} \times \Z_{2^4}^{\ell_2n} \times \Z_{2^2}^{\ell_1n}) \\
\cong & \Z_{2^{3}}^{(2 M+\sum_{j=1}^w v_j+h_2)n} \times\Z_2^{h_1n}.
\end{align*}
By \cref{lem-quotient}(2),
\begin{align}
\begin{split}
\label{eqn-admissible-3}
&\eta_1(\Z_{2^{2s_1}}^{xn} \times \Z_{2^6}^{2 M n} \times \Z_{2^4}^{\ell_2n} \times \Z_{2^2}^{\ell_1n} \times \Z_{2^{3}}^{(\sum_{j=1}^w v_j+h_2)n} \times\Z_2^{h_1n})\\
\le &\eta_1(\Z_{2^{2s_1}}^{xn} \times \Z_{2^3}^{2 M n} \times \Z_{2^4}^{\ell_2n} \times \Z_{2^2}^{\ell_1n}) \eta_1(\Z_{2^{3}}^{(2 M+\sum_{j=1}^w v_j+h_2)n} \times\Z_2^{h_1n}).
\end{split}
\end{align}
Combining \cref{cor-directproduct}(2), \cref{lem-Z22s}, and \cref{cor-peven}, for sufficiently large $n$,
\begin{equation}
\label{eqn-admissible-4}
\eta_1(\Z_{2^{2s_1}}^{xn} \times \Z_{2^3}^{2 M n} \times \Z_{2^4}^{\ell_2n} \times \Z_{2^2}^{\ell_1n})=2^{(s_1x+3 M+2\ell_2+\ell_1)n}(1+o_n(1)).
\end{equation}

Since $G_2$ has square order, $\sum_{j=1}^w v_j +h_2+h_1 \equiv 0 \pmod2$. Together with $2M+\sum_{j=1}^w v_j +h_2 \ge h_1$, we can write $2M+\sum_{j=1}^w v_j +h_2-h_1=2q$ for some $q \ge 0$. Note that $$(\Z_{2^{3}}^{(2M+\sum_{j=1}^w v_j+h_2)n} \times\Z_2^{h_1n})/(\Z_{2^3}^{h_1n} \times \Z_2^{h_1n}) \cong \Z_{2^3}^{2qn}.$$ By \cref{lem-quotient}(2),
\begin{equation}
\label{eqn-admissible-5}
\eta_1(\Z_{2^{3}}^{(2M+\sum_{j=1}^w v_j+h_2)n} \times\Z_2^{h_1n})=\eta_1(\Z_{2^3}^{h_1n} \times \Z_2^{h_1n})\eta_1(\Z_{2^3}^{2qn}).
\end{equation}
By \cref{lem-Z8n2n,cor-peven}, for sufficiently large $n$,
\begin{align}
\label{eqn-admissible-6}
\eta_1(\Z_{2^3}^{h_1n} \times \Z_2^{h_1n})=2^{2h_1n}(1+o_n(1)), \quad \eta_1(\Z_{2^3}^{2qn})=2^{3qn}(1+o_n(1)).
\end{align}
Combining \cref{eqn-admissible-1,eqn-admissible-2,eqn-admissible-3,eqn-admissible-4,eqn-admissible-5,eqn-admissible-6}, for sufficiently large $n$,
$
\eta_{1}(G_2^{n}) \le (1+o_{n}(1))|G_2|^{\frac{n}{2}}.
$

Since $G=G_2 \times H \times H$, where $H$ has odd order, by \cref{lem-podd} and \cref{cor-directproduct}(2), we have
\[
\eta_{1}(G^{n})\leq \eta_{1}(G_{2}^{n}) \eta_{1}(H^{2 n}) \leq(1+o_{n}(1))|G_{2}|^{n/2}(1+o_{n}(1))|H|^{n}=(1+o_{n}(1))|G|^{n/2}.
\]
Since $\eta_1(G^{n}) \ge |G|^{\frac{n}{2}}$, then $\eta_{1}(G^{n})=(1+o_{n}(1))|G|^{\frac{n}{2}}$ as required. Finally, the existence of the limit $\lim_{n \to \infty} \frac{\eta_g(G^n)}{\sqrt{g}|G|^{n/2}}$ follows from \cref{prop-limits}.

\section{Conclusion}\label{sec:conclusion}

Motivated by recent works on $g$-Sidon sets and $g$-difference bases, we study $g$-additive bases and $g$-difference bases in groups of the form $G^n$, where $G$ is a finite abelian group and $n \ge 1$. Our central question is to estimate or determine $\nu_g(G^n)$ and $\eta_g(G^n)$, the minimal sizes of a $g$-additive basis and a $g$-difference basis in $G^n$, respectively. Using a constructive approach, we build $g$-additive and $g$-difference bases in $G^n$; these constructions, in turn, yield asymptotic information about $\nu_g(G^n)$ and $\eta_g(G^n)$ in suitable limiting regimes.

We conclude by posing three open questions that may be of interest:

\begin{itemize}
\item[(1)] In \cref{thm-admissible}, we have shown the existence of the limit $\lim_{n\rightarrow \infty} \frac{\eta_g(G^n)}{\sqrt{g}|G|^{n/2}}$ for certain finite abelian group $G$ and all positive integers $g$. When $g$ goes to infinity, we know the asymptotic behavior of these limits from \cref{thm-limitg}. However, we are not able to determine these limits explicitly. What are the exact values of these limits?
\item[(2)] For $g$-additive bases, is it possible to establish results similar to \cref{thm-admissible}, establishing the existence of $\lim_{n\rightarrow \infty} \frac{\nu_g(G^n)}{\sqrt{g}|G|^{n/2}}$ for each fixed positive integer $g$? 
\item[(3)] Is it possible to weaken the weakly admissible and admissible conditions in \cref{thm-weaklyadmissible} and \cref{thm-admissible}? 
\end{itemize}

\section*{Acknowledgments}
Shuxing Li's research was supported by the U.S. National Science Foundation Grant DMS-2452236. C. H. Yip thanks University of Delaware for hospitality during his visit with the first author in March 2025, where this project was initiated. C. H. Yip also thanks Ernie Croot for helpful discussions. 
\bibliographystyle{abbrv}
\bibliography{main}
\appendix
\section{Improved and explicit upper bounds on $\nu_g(\Z_{p^s}^{2n})$ and $\eta_g(\Z_{p^s}^{2n})$}\label{appendix}

In this appendix, we present improved upper bounds on $\nu_g(\Z_{p^s}^{2n})$ and $\eta_g(\Z_{p^s}^{2n})$. As an illustrative example, we dive into more details concerning the upper bounds on $\nu_g(\Z_{p^s}^{2n})$ and $\eta_g(\Z_{p^s}^{2n})$ for $s \ge 1$, $n \ge 1$, and $g \in [6]$.      

The following lemma is a strengthening of \cref{cor-PCPextension}.

\begin{lem}
\label{lem-PCP}
Let $p$ be a prime and $s, n \ge 1$. Let $2 \le k \le \min\{p^n,\frac{p^{sn}}{2} \}$. Then we have
$$
\nu_{k^{2}-k}(\Z_{p^{s}}^{2 n}) \le   k (p^{sn}-1), \quad \eta_{k^{2}-k}(\Z_{p^{s}}^{2 n}) \le k (p^{sn}-1).
$$
\end{lem}

\begin{proof}
Let $R=\GR(p^{s}, n)$. Note that $(R \times R,+) \cong \Z_{p^s}^{2n}$. We work on $R \times R$ in order to construct the corresponding additive basis and difference basis in $\Z_{p^s}^{2n}$. Let $\cT$ be the Teichmuller system of $R$ and $R^*=R \sm \{0\}$. 

We first construct a $(k^2-k)$-difference basis of size $k(p^{sn}-1)$ in $(R \times R,+)$, where $2 \le k \le p^{n}$. By \cref{fact-GR}(3), we can assume that $\{\alpha_{1}, \ldots, \alpha_{k}\}$ be a subset of $\cT$, such that $\al_i-\al_j \in U(R)$ for $1 \le i, j \le k$ and $i \ne j$. For each $1 \le i \le k$, let
$$
S_i=\{(x, \alpha_{i} x) \mid x \in R^{*}\}.
$$ 
and
\begin{equation}
\label{eqn-S}
S=\bigcup_{i=1}^k S_i=\bigcup_{i=1}^{k}\{(x, \alpha_{i} x) \mid x \in R^{*}\}.
\end{equation}
Then $|S| \le k(p^{sn}-1)$ and we claim that $S$ is a $k(k-1)$-difference basis in $(R \times R,+) \cong \Z_{p^s}^{2n}$. For $(a,b) \in R \times R$ and $1 \le i, j \le k$, by solving $(a, b)=(x_{1}, \alpha_{i} x_{1})-(x_{2}, \alpha_{j} x_{2})$, we have the following two cases.

Case (a): $b-\alpha_{\ell} a \neq 0$ for each $1 \leq \ell \leq k$. Then $(a,b)$ cannot be expressed as a difference of elements from $\{ (x,\al_{\ell} x) \mid x \in R^*\}$, where $1 \le \ell \le k$. On the other hand, for distinct $i$ and $j$ with $1 \le i,j \le k$, $(a, b)=(x_{1}, \alpha_{i} x_{1})-(x_{2}, \alpha_{j} x_{2})$ is equivalent to
$$
\begin{cases}
x_1=\frac{b-\al_ja}{\al_i-\al_j} \in R^*, \\
x_2=\frac{b-\al_ia}{\al_i-\al_j} \in R^*.
\end{cases}
$$ 
Therefore, for distinct $i$ and $j$, $(a,b)$ can be expressed as exactly one difference between elements in $S_i$ and elements in $S_j$. Hence, $(a,b)$ can be expressed as exactly $k(k-1)$ differences from $S$.

Case (b): $b-\alpha_{\ell} a=0$ for exactly one $\ell$ with $1 \leq \ell \leq k$. Similar to Case (a), $(a,b)$ can be expressed as exactly $(k-1)(k-2)$ differences from $\bigcup_{\substack{1 \le i \le k \\ i \ne \ell}} S_i$. Moreover, within $S_{\ell}=\{(x, \alpha_{\ell} x) \mid x \in R^{*}\}$, we have $(a, b)=(a, \alpha_{\ell} a)=(a+c, \alpha_{\ell}(a+c))-(c, \alpha_{\ell} c)$. Note that $(a+c, \alpha_{\ell}(a+c)) \in S_{\ell}$ and $(c, \alpha_{\ell} c) \in S_{\ell}$ if and only if $c \in R \backslash\{0,-a\}$. Hence, $(a, b)$ can be expressed as exactly $p^{sn}-2$ differences from $S_{\ell}$. Consequently, $(a, b)$ can be expressed as exactly $p^{sn}-2+(k-1)(k-2)$ differences from $S$.

Note that $k \le \frac{p^{sn}}{2}$ implies $p^{sn}-2+(k-1)(k-2) \ge k(k-1)$. Consequently, $S$ is a $k(k-1)$-difference basis in $(R \times R,+) \cong \Z_{p^s}^{2n}$, and therefore, $\eta_{k^{2}-k}(\Z_{p^{s}}^{2 n}) \leq k(p^{sn}-1)$. 

Next, we proceed to construct a $(k^2-k)$-additive basis of size $k(p^{sn}-1)$ in $(R \times R,+)$. Note that for each $1 \le i \le k$, we have $-S_i=\{ (-x,-\alpha_ix) \mid x \in \R^* \}=S_i$. Therefore, the set $S$ in \Cref{eqn-S} satisfies $-S=\{-x \mid x \in S\}=S$ and is also a $(k^2-k)$-additive basis of size $k(p^{sn}-1)$ in $(R \times R,+)$, as required.
\end{proof}

\begin{rem}
\label{rem-small}
The condition $2 \le \min\{p^n,\frac{p^{sn}}{2} \}$ in \Cref{lem-PCP} excludes two cases where $(p,s,n) \in \{ (2,1,1), (3,1,1) \}$. In fact, the upper bounds in \Cref{lem-PCP} do not hold in both cases. Specifically, 
\begin{align*}
(\nu_1(\Z_2^2),\nu_2(\Z_2^2),\nu_3(\Z_2^2),\nu_4(\Z_2^2))=&(3,3,4,4), \\
(\eta_1(\Z_2^2),\eta_2(\Z_2^2),\eta_3(\Z_2^2),\eta_4(\Z_2^2))=&(3,3,4,4),
\end{align*}
and
\begin{align*}
(\nu_1(\Z_3^2),\nu_2(\Z_3^2),\nu_3(\Z_3^2),\nu_4(\Z_3^2),\nu_5(\Z_3^2),\nu_6(\Z_3^2),\nu_7(\Z_3^2),\nu_8(\Z_3^2))=&(4,5,6,7,7,8,8,9), \\
(\eta_1(\Z_3^2),\eta_2(\Z_3^2),\eta_3(\Z_3^2),\eta_4(\Z_3^2),\eta_5(\Z_3^2),\eta_6(\Z_3^2),\eta_7(\Z_3^2),\eta_8(\Z_3^2))=&(4,5,6,7,7,8,8,9).
\end{align*}
\end{rem}

\medskip

Below, we explicitly compute upper bounds on $\nu_g(\Z_{p^s}^{2n})$ and $\eta_g(\Z_{p^s}^{2n})$ for $g \in [6]$, by combining \Cref{cor-directproduct}, \Cref{lem-podd}, and \Cref{lem-PCP}. We exclude the two cases $(p,s,n) \in \{ (2,1,1), (3,1,1) \}$ where the corresponding groups have very small sizes and have been handled in \Cref{rem-small}. 

%Let $p$ be a prime. Let $s \ge 1$ and $n \ge 1$ with $(p,s,n) \not\in \{ (2,1,1), (3,1,1) \}$. 

We consider upper bounds on $\nu_g(\Z_{p^s}^{2n})$.
\begin{itemize}
\item For $n \ge 1$, by \Cref{lem-PCP},
\begin{align}
\begin{split}
\nu_1(\Z_{p^s}^{2n}) & \le \nu_2(\Z_{p^s}^{2n}) \le 2(p^{sn}-1) \\
\nu_3(\Z_{p^s}^{2n}) & \le \nu_4(\Z_{p^s}^{2n}) \le \nu_5(\Z_{p^s}^{2n}) \le \nu_6(\Z_{p^s}^{2n}) \le 3(p^{sn}-1)
\label{eqn-rem-1}
\end{split}
\end{align}
\item For $n \ge 2$, by \Cref{cor-directproduct} and \Cref{eqn-rem-1},
\begin{align}
\begin{split}
\nu_1(\Z_{p^s}^{2n}) & \le \nu_2(\Z_{p^s}^{2n}) \le \nu_3(\Z_{p^s}^{2n}) \le \nu_4(\Z_{p^s}^{2n}) \le \nu_2(\Z_{p^s}^{2})\nu_2(\Z_{p^s}^{2(n-1)}) \le 4(p^s-1)(p^{s(n-1)}-1) \\
\nu_5(\Z_{p^s}^{2n}) & \le \nu_6(\Z_{p^s}^{2n}) \le \nu_3(\Z_{p^s}^{2})\nu_2(\Z_{p^s}^{2(n-1)}) \le 6(p^s-1)(p^{s(n-1)}-1)
\label{eqn-rem-2}
\end{split}
\end{align}
\item For $n \ge 3$, by \Cref{cor-directproduct} and \Cref{eqn-rem-1,eqn-rem-2},
\begin{align}
\begin{split}
\nu_1(\Z_{p^s}^{2n}) \le & \nu_2(\Z_{p^s}^{2n}) \le \nu_3(\Z_{p^s}^{2n}) \le \nu_4(\Z_{p^s}^{2n}) \le \nu_5(\Z_{p^s}^{2n}) \le \nu_6(\Z_{p^s}^{2n}) \\  \le & \nu_3(\Z_{p^s}^4)\nu_2(\Z_{p^s}^{2(n-2)}) \\
    \le & \min\{ 6(p^{2s}-1)(p^{s(n-2)}-1), 8(p^s-1)^2(p^{s(n-2)}-1) \} \label{eqn-rem-3}
%    \\
%    \le & \begin{cases}
%              6(p^{2s}-1)(p^{s(n-2)}-1) & \mbox{if $p^s \ge 7$} \\
%              8(p^s-1)^2(p^{s(n-2)}-1)  & \mbox{if $p^s < 7$}
%          \end{cases}
\end{split}
\end{align}
\end{itemize}
If needed, we can further optimize the upper bounds on $\nu_g(\Z_{p^s}^{2n})$, $g \in [6]$, by combining the upper bounds listed in \Cref{eqn-rem-1,eqn-rem-2,eqn-rem-3}. 

Next, we consider the upper bounds on $\eta_g(\Z_{p^s}^{2n})$ for  $p$ being an odd prime. 
\begin{itemize}
\item For $n \ge 1$, by \Cref{lem-podd} and \Cref{lem-PCP},
\begin{align}
\begin{split}
\eta_1(\Z_{p^s}^{2n}) & \le p^{sn}+9p^{(s-\frac{1}{2})n} \\
\eta_1(\Z_{p^s}^{2n}) & \le \eta_2(\Z_{p^s}^{2n}) \le 2(p^{sn}-1) \\
\eta_3(\Z_{p^s}^{2n}) & \le \eta_4(\Z_{p^s}^{2n}) \le \eta_5(\Z_{p^s}^{2n}) \le \eta_6(\Z_{p^s}^{2n}) \le 3(p^{sn}-1)
\label{eqn-rem-4}
\end{split}
\end{align}
\item For $n \ge 2$, by \Cref{cor-directproduct} and \Cref{eqn-rem-4},
\begin{align}
\begin{split}
\eta_1(\Z_{p^s}^{2n}) & \le \eta_2(\Z_{p^s}^{2n}) \le \eta_2(\Z_{p^s}^{2})\eta_1(\Z_{p^s}^{2(n-1)}) \le 2(p^{s}-1)(p^{s(n-1)}+9p^{(s-\frac{1}{2})(n-1)}) \\
\eta_1(\Z_{p^s}^{2n}) & \le \eta_2(\Z_{p^s}^{2n}) \le \eta_3(\Z_{p^s}^{2n}) \le \eta_4(\Z_{p^s}^{2n}) \le \eta_2(\Z_{p^s}^{2})\eta_2(\Z_{p^s}^{2(n-1)}) \le 4(p^s-1)(p^{s(n-1)}-1) \\
\eta_5(\Z_{p^s}^{2n}) & \le \eta_6(\Z_{p^s}^{2n}) \le \eta_3(\Z_{p^s}^{2})\eta_2(\Z_{p^s}^{2(n-1)}) \le 6(p^s-1)(p^{s(n-1)}-1)
\label{eqn-rem-5}
\end{split}
\end{align}
\item For $n \ge 3$, by \Cref{cor-directproduct} and \Cref{eqn-rem-4,eqn-rem-5},
\begin{align}
\begin{split}
\eta_1(\Z_{p^s}^{2n}) \le & \eta_2(\Z_{p^s}^{2n}) \le \eta_3(\Z_{p^s}^{2n}) \le \eta_4(\Z_{p^s}^{2n}) \le \eta_5(\Z_{p^s}^{2n}) \le \eta_6(\Z_{p^s}^{2n}) \\  \le & \eta_3(\Z_{p^s}^4)\eta_2(\Z_{p^s}^{2(n-2)}) \\
    \le & \min\{ 6(p^{2s}-1)(p^{s(n-2)}-1), 8(p^s-1)^2(p^{s(n-2)}-1) \}
\label{eqn-rem-6}
\end{split}
\end{align}
\end{itemize}
If needed, we can further optimize the upper bounds on $\eta_g(\Z_{p^s}^{2n})$, $p$ odd prime, $g \in [6]$, by combining the upper bounds listed in \Cref{eqn-rem-4,eqn-rem-5,eqn-rem-6}. 

\end{document}